\documentclass[11pt,letterpaper]{amsart}
\usepackage{amsmath}
\usepackage{amssymb}
\usepackage{amsthm}
\usepackage{hyperref}
\usepackage[noabbrev,capitalize]{cleveref}
\usepackage{mathrsfs}
\usepackage{mathtools}
\usepackage{slashed}
\usepackage{tikz-cd}
\usepackage[shortlabels]{enumitem}

\setenumerate{label=(\roman*)}

\DeclareMathOperator{\Ric}{Ric}

\newcommand{\lp}{\langle}
\newcommand{\rp}{\rangle}

\newcommand{\mA}{\mathcal{A}}

\newcommand{\mH}{\mathcal{H}}

\def\sideremark#1{\ifvmode\leavevmode\fi\vadjust{\vbox to0pt{\vss
 \hbox to 0pt{\hskip\hsize\hskip1em
 \vbox{\hsize3cm\tiny\raggedright\pretolerance10000
 \noindent #1\hfill}\hss}\vbox to8pt{\vfil}\vss}}}

\newtheorem{theorem}{Theorem}[section]
\newtheorem{proposition}[theorem]{Proposition}

\newtheorem{corollary}[theorem]{Corollary}

\theoremstyle{definition}
\newtheorem{definition}[theorem]{Definition}

\theoremstyle{remark}
\newtheorem{remark}[theorem]{Remark}

\numberwithin{equation}{section}

\makeatletter
\@namedef{subjclassname@2020}{\textup{2020} Mathematics Subject Classification}
\makeatother

\begin{document}

\title[m-intermediate curvature]{intermediate curvature and splitting theorem}

\author{Jingche Chen}
\address{Department of Mathematical Sciences, Tsinghua University, 100084, Beijing, China}
\email{cjc23@mails.tsinghua.edu.cn}

\author{Han Hong}
\address{Department of Mathematics and statistics \\ Beijing Jiaotong University \\ Beijing \\ China, 100044}
\email{hanhong@bjtu.edu.cn}

\begin{abstract}
In this paper, we prove several rigidity results for complete noncompact manifolds with nonnegative intermediate curvatures. We show that when either $3\leq n\leq 5$, $1\leq m\leq n-1$, or $6\leq n\leq 7$, $m\in \{1,n-1,n-2\}$, any manifold of the topological type $M^{n-m}\times \mathbb{T}^{m-1}\times \mathbb{R}$  with nonnegative $m$-intermediate curvature is isometrically covered by the canonical product $M\times \mathbb{R}^m$. We also construct smooth metrics on $M^{n-m}\times \mathbb{T}^{m-1}\times \mathbb{R}$ with uniformly positive $m$-intermediate curvature for $6\leq n\leq 7$, $2\leq m\leq n-3$. This proves that the algebraic condition $m^2-mn+m+n>0$ from \cite{chenshuli_end} is sharp. The proof is based on a new recursion theorem for spectral intermediate curvatures and cylindrical splitting theorems. In particular, when $m=n-1$, this provides a new proof of some results by Chodosh--Li \cite{chodoshlisoapbubble} and Zhu \cite{zhu-splitting}. Moreover, the recursion theorem can be used to reprove the result of Brendle--Hirsch--Johne \cite{brendlegeroch'sconjecture}.
\end{abstract}

\maketitle

\section{introduction}
The relation between topology and curvature has long been a fascinating topic in differential geometry. The classical Gauss--Bonnet theorem implies that a non-spherical closed surface cannot admit positive Gaussian curvature, and if it admits nonnegative Gaussian curvature, it must be a flat torus. On the other hand, the Bonnet--Myers theorem states that an $n$-dimensional closed manifold with positive Ricci curvature must have finite fundamental group, and if the manifold is orientable, then $H_{n-1}$ with $\mathbb{Z}$-coefficient is trivial.

The well-known Geroch conjecture asks whether $\mathbb{T}^n$ admits a metric of positive scalar curvature. This conjecture was resolved by Schoen--Yau \cite{Schoen-Yau-PSC} for $3 \leq n \leq 7$ using minimal hypersurfaces, and by Gromov--Lawson \cite{Gromov-Lawson-ann} using spinors for all dimensions. Recent work by Chodosh--Li \cite{chodoshlisoapbubble} (see also Gromov \cite{gromov-apherical-conjecture}) proves that closed $n$-dimensional aspherical manifolds for $n = 4,5$ cannot admit positive scalar curvature. Brendle--Hirsch--Johne \cite{brendlegeroch'sconjecture} found curvature obstructions for manifolds of topological type $M^{n-m}\times \mathbb{T}^m$ and proved that such manifolds cannot admit positive $m$-intermediate curvature. Their result connects the Bonnet--Myers theorem or Bochner's formula to the Geroch conjecture. In most cases, a non-existence theorem is accompanied by a rigidity result. For example, by the work of Kazdan \cite{kazdan-deformation} and the Cheeger--Gromoll splitting theorem, $\mathbb{T}^n$ with nonnegative scalar curvature must be flat. Chu--Kwong--Lee \cite{ChuKwongLeerigidity} proved that $M^{n-m}\times \mathbb{T}^m$ for $3 \leq n \leq 5$ with nonnegative $m$-intermediate curvature is isometrically covered by $E \times \mathbb{R}^m$ with a closed manifold $E$ satisfying $\operatorname{Ric}_E \geq 0$. This was later extended by Xu \cite{XuTrans} to the case $n = 6$.

There have also been rich developments of nonexistence and rigidity results in noncompact settings. For instance, Chodosh--Li \cite{chodoshlisoapbubble} (see also \cite{lesourd-unger-yau} for $n=3$) proved the Geroch conjecture for manifolds with arbitrary ends and its rigidity; namely, they showed that the connected sum $\mathbb{T}^n\# X$ for $3\leq n\leq 7$ with nonnegative scalar curvature must be flat, where $X$ is an arbitrary manifold. For the same range of $n$, using a different idea, Zhu \cite{zhu-splitting} proved that a manifold of nonnegative scalar curvature admitting a smooth proper map to $\mathbb{T}^{n-1}\times \mathbb{R}$ with nonzero degree must be flat. For spin manifolds, Wang--Zhang \cite{Wang-Zhang-geroch} showed that for any $n$, $\mathbb{T}^n\# X$ cannot admit positive scalar curvature when $X$ is spin. To obtain a noncompact version of the result of Brendle--Hirsch--Johne, Chen \cite{chenshuli_end} proved
\begin{theorem}\cite{chenshuli_end}\label{chenshuli theorem}
Assume either $3 \le n \le 5$, $1 \le m \le n-1$ or $6 \le n \le 7$, $m \in \{1, n-2, n-1\}$. 
Let $N^n$ be a closed manifold of dimension $n$, and suppose that there exists a closed manifold $M^{n-m}$ and a map 
\[
F : N^n \to M^{n-m} \times T^m
\]
with non-zero degree. Then for any $n$-manifold $X$, the connected sum $N \# X$ does not admit a complete metric of positive $m$-intermediate curvature.
\end{theorem}

To better explain this result, we recall the definition of \textit{$m$-intermediate curvature} defined by Brendle-Hirsch-Johne \cite{brendlegeroch'sconjecture}.

\begin{definition}

    Suppose $(N,g)$ is an $n$-dimensional manifold. For given orthonormal vectors $\{e_{1},...,e_{m}\}$ in $T_pN$ at the point $p\in N$, one can extend them to an orthonormal basis $\{e_{1},...,e_{n}\}$ of $T_{p}N$. 
    The $m$-intermediate curvature $C_{m}$ of the orthornormal vectors $\{e_{1},...,e_{m}\}$ is defined by
    \[
    C^N_{m}(e_{1},...,e_{m}) = \sum_{p=1}^{m}\sum_{q=p+1}^{n}R^N(e_{p},e_{q},e_{p},e_{q}).\]
In particular, $C^N_1(e_1)=\Ric_N(e_1,e_1)$ and $2C^N_{n-1}(e_1,\cdots,e_{n-1})=R_N$ where $R_N$ is the scalar curvature of $N$. Let 
\[C^N_m(p):=\min\{C^N_m(e_1,\cdots,e_m)| \{e_1,\cdots,e_m\}\ \text{are orthonormal in}\ T_pN\}.\]
We say $(N,g)$ has nonnegative (positive) $m$-intermediate curvature if $C^N_m(p)\geq 0\ (>0)$ for all $p\in N.$ When $m=1$, we also use $\Ric_N:=\min\{\Ric_N(v,v):v\in T_pN\}$ to denote $C_1^N.$
\end{definition}

We now explain the assumptions in Theorem \ref{chenshuli theorem}. The condition $n \leq 7$ is due to the regularity of minimal hypersurfaces (also due to some counterexamples of Xu \cite{XuTrans}), and we will always assume this throughout the paper unless otherwise stated. The range of $m$ and $n$ can be rephrased into two algebraic inequalities:
\[
m^2 - mn + 2n - 2 > 0 \quad \text{and} \quad m^2 - mn + m + n > 0,
\]
where the first inequality first appeared in \cite{brendlegeroch'sconjecture} and it is always satisfied by $3\leq n\leq 7$ and $m\leq n-1$. It is the second inequality that causes the incompleteness of the admissible pairs $(n,m)$ in Theorem \ref{chenshuli theorem}. In fact, Chen \cite{chenshuli_end} considered a cyclic cover $\tilde{N}$ of $N$ by opening a $\mathbb{S}^1$-factor and used the $\mu$-bubble to reduce the problem in the noncompact setting to the compact setting, which requires the second inequality. This same inequality also appears in the rigidity result of Chu--Kwong--Lee \cite{ChuKwongLeerigidity} when they employed the $\mu$-bubble. 
Note that a feature of $\tilde{N}$ is that it admits a smooth proper map to $M^{n-m}\times\mathbb{T}^{m-1}\times\mathbb{R}$ with nonzero degree. Thus Chen actually proved that such $\tilde{N}$ can not admit positive $m$-intermediate curvature in the given range of $m,n$. It remains an open question whether the condition $m^2-mn+m+n>0$ can be dropped so that nonexistence result holds for all $3 \leq n \leq 7$ and $1 \leq m \leq n-1$ (see Chen \cite{chenshuli_end}).

The goal of this paper is twofold. First, we establish a rigidity result. Second, we answer this question in the negative. The main results are as follows.

\begin{theorem}\label{splitting of nonnegative C_m: main in the introduction}
Assume either \(3 \leq  n \leq  5\), \(1 \leq  m \leq  n-1\) or \(6 \leq  n \leq  7\), \(m\in \{1,n-2,n-1\}\). Let \((N^n,g)\) be an orientable complete noncompact Riemannian manifold with nonnegative $m$-intermediate curvature, which admits a smooth proper map \(f : N \to M^{n-m}\times \mathbb{T}^{m-1} \times \mathbb{R}\) with nonzero degree. Then \((N,g)\) is isometric to the Riemannian product of a closed manifold $E$, flat $(m-1)$-torus and the real line where $E$ has nonnegative Ricci curvature.
\end{theorem}

 Note that when $m = 1$, Theorem \ref{splitting of nonnegative C_m: main in the introduction} follows directly from the classical Cheeger--Gromoll splitting theorem (in this case, there is even no restriction on $n$). Theorem \ref{splitting of nonnegative C_m: main in the introduction} recovers the results of Chodosh--Li \cite{chodoshlisoapbubble} and Zhu \cite{zhu-splitting} for the case $m = n-1$.
 
 \begin{theorem}\label{counterexample of splitting: main in the introduction}
    When $6\leq n\leq 7$ and $2\leq m\leq n-3$, there exists a complete smooth metric on $\mathbb{S}^{n-m}\times T^{m-1}\times \mathbb{R}$ with uniformly positive $m$-intermediate curvature.
\end{theorem}

When $m=2$, Theorem \ref{counterexample of splitting: main in the introduction} recovers the result of Xu \cite{XuTrans}. Above two theorems are not independent. Indeed, the proof of the first theorem inspires the construction of examples in the second theorem. 

\subsection{Idea of proof}
 We first prove a recursion theorem for spectral $m$-intermediate curvatures. Roughly speaking, it states that if the ambient manifold has nonnegative spectral $m$-intermediate curvature, then a certain weighted minimizer admits nonnegative spectral $(m-1)$-curvature (see Theorem \ref{spectral inheritance}). The motivation for this result is as follows. The proof of the Geroch conjecture \cite{Schoen-Yau-PSC} makes use of the fact that an area-minimizer in a manifold of positive scalar curvature itself admits pointwisely positive scalar curvature. On the other hand, the proof of the generalized version \cite{brendlegeroch'sconjecture} introduces stable weighted minimal compact slicings and sums all stability inequalities of slicings on the bottom slice to obtain
\[
0 \geq \int_{\text{bottom slice}} C_m + \mathcal{V}_1 + \mathcal{V}_2 + \cdots
\]
where $\mathcal{V}_i$ are nonnegative expressions involving the second fundamental forms of $\Sigma_i$. One may expect a similar inheritance property of $m$-intermediate curvature as in the proof of Geroch conjecture.  With the recursion theorem, we can reprove the result in \cite{brendlegeroch'sconjecture} (see Theorem \ref{a new proof of brendle's result}). In particular, this gives a spectral-sense proof for Georoch conjecture, since on the last slice (closed curve), there is no positive strictly concave function.

The splitting theorem of noncompact manifold is often difficult because  the the minimizer could drift off. The useful technique is to apply $\mu$-bubble and restrain the minimizer in a bounded region. 
We shall use a different method in this paper. The key idea of the proofs of the main results is illustrated in Figure \ref{picture of proof}. Using the special topological structure of $M$, we construct a chain of (noncompact) cylindrical manifolds
\[
\Sigma_{m-1} \subset \cdots \subset \Sigma_{\ell} \subset \Sigma_0 = N
\]
where each slice admits nonnegative spectral $\ell$-intermediate curvature. For the last slice $\Sigma_{m-1}$, the manifold has at least two ends and satisfies
\[
-\frac{2m-2}{m} \Delta_{\Sigma_{m-1}} u_{m-1} + C_{1}^{\Sigma_{m-1}} \cdot u_{m-1} = 0.
\]
Using a spectral splitting result by Antonelli--Xu and Catino--Mari--Mastrolia--Roncoroni (see Theorem \ref{splitting of nonnegative spectral ric}), one deduces that $\Sigma_{m-1}$ splits if
\[
\frac{2m-2}{m} < \frac{4}{n-m}.
\]
This inequality is exactly the algebraic inequality $m^2 - mn + m + n > 0$.

 \begin{figure}[htbp]\label{picture of proof}
\centering
\[
\begin{array}{c|c}
\hline
\multicolumn{2}{c}{\text{Slicings and spectral curvatures}\ (0\leq \ell<m<n\leq 7)} \\ \hline 
N^n \approx M^{n-m} \times \mathbb{T}^{m-1} \times \mathbb{R} & C_m^N \geq 0 \\
\Downarrow & \downarrow \\
\Sigma_{1}^{n-1} \approx M^{n-m} \times \mathbb{T}^{m-2} \times \mathbb{R} & -\Delta_{\Sigma_1} u_1 + C_{m-1}^{\Sigma_1} \cdot u_1 = 0 \\
\vdots & \vdots \\
\Sigma_{\ell}^{n-\ell} \approx M^{n-m} \times \mathbb{T}^{m-\ell-1} \times \mathbb{R} & -\frac{2\ell}{\ell+1} \Delta_{\Sigma_\ell} u_\ell + C_{m-\ell}^{\Sigma_\ell} \cdot u_\ell = 0 \\
\vdots & \vdots \\
\Sigma_{m-2}^{n-m+2} \approx M^{n-m} \times \mathbb{S}^1 \times \mathbb{R} & -\frac{2m-4}{m-1} \Delta _{\Sigma_{m-2}} u_{m-2} + C_{2}^{\Sigma_{m-2}} \cdot u_{m-2} = 0 \\
\Downarrow & \downarrow \\
\Sigma_{m-1}^{n-m+1} \approx M^{n-m} \times \mathbb{R} & -\frac{2m-2}{m} \Delta_{\Sigma_{m-1}} u_{m-1} + C_{1}^{\Sigma_{m-1}} \cdot u_{m-1} = 0\\ [7pt]
\hline
\end{array}
\]
\caption{$\Downarrow$  means minimization with respect to $\mathcal{A}_{2\ell/(\ell+1)}(u)$; $\approx$ means the existence of a nonzero degree smooth proper map from left to right; $u_i$ are all positive functions.}

\end{figure}
 
 From here, the proof of the rigidity result and the construction of counterexamples is clear. We use the interesting concept of ``reverse engineering'' introduced by Xu \cite{XuTrans}. For rigidity, we prove a spectral cylindrical splitting theorem (see Theorem \ref{spectral splitting theorem of C_2}). This resembles a higher-dimensional version of the splitting theorem proved by Chodosh--Eichmair--Moraru \cite{chodoshsplittingtheorem}, in which they assume the existence of an absolutely area-minimizing cylinder. However, due to the topological structure, we do not need to prove a direct higher-dimensional generalization of their result; rather, we focus on the existence of a minimizing cylinder and the non-negativity of the ambient Ricci curvature. Nevertheless, it would be interesting to obtain such a result in full generality. By applying the spectral splitting theorem and Cheeger-Gromoll splitting theorem iteratively, we can split the top manifold $N$ completely. On the other hand, the construction of a non-splitting metric on $N$ with nonnegative $m$-intermediate curvature begins with non-splitting examples on the bottom slice $\Sigma_{m-1}$ when $\Sigma_{m-1}$ itself does not split. Luckily, Xu \cite{XuTrans} has provided such examples.

 In subsection \ref{section4}, we also give an application of recursion theorem on the diameter estimate of objects in $k$-th  homology class.

\subsection{A comparison between various proofs of splitting theorems}
As mentioned before, Chodosh--Li \cite{chodoshlisoapbubble} and Zhu \cite{zhu-splitting} proved splitting theorems on open manifolds with nonnegative scalar curvature. One may wonder whether their ideas can be directly generalized to $m$-intermediate curvatures. In \cite{chodoshlisoapbubble}, they rely on a theorem of Kazdan \cite{kazdan-deformation}, the $\mu$-bubble method, and the Cheeger--Gromoll splitting theorem. In contrast, Zhu \cite{zhu-splitting} avoids the Cheeger--Gromoll splitting theorem but instead uses compact foliations as in \cite{BrayBrenleNevesrigidity}. The key step in his approach is to prove the compactness of the minimal limit of a sequence of compact $\mu$-bubbles. This step still implicitly uses Kazdan's theorem or the theory from conformal geometry. However, such techniques seem unlikely to extend to the setting of intermediate curvatures. This forces us to develop a new idea that works both for scalar curvature and intermediate curvatures.

\vskip.1cm

\subsection{Final remark}
All the results presented in this article use nonnegative pointwise $m$-intermediate curvature on the manifold $N\approx M\times \mathbb{T}^{n-1}\times \mathbb{R}$. One can also define nonnegative spectral $m$-intermediate curvature on $N$ and prove that for $3\leq n\leq 7$ and $m\in \{1, n-2, n-1\}$, such a topological manifold $N$ must split if it admits nonnegative spectral $(\gamma,m)$-intermediate curvature for $0<\gamma<C(n,m)$, where $C(n,m)$ is a constant depending on $n$ and $m$. In particular, $C(n,n-1)=\frac{2n}{n-1}$. This special case was recently obtained by Chai and Sun \cite{chai-sun-spectralscalarsplitting} by generalizing arguments of Chodosh–Li and Zhu to the spectral setting.

\vskip.2cm
The structure of the paper is as follows. In Section \ref{section2}, we introduce spectral $m$-intermediate curvature and prove the crucial recursion theorem (Thereom \ref{spectral inheritance}), we use it to reprove the result of Brendle-Hirsch-Johne. In section \ref{section3}, we prove a lower dimensional cylindrical splitting theorem, and then use it to prove our first main theorem (Theorem \ref{splitting of nonnegative C_m: main in the introduction}).  In section \ref{section4}, we first construct an example that proves our second main theorem (Theorem \ref{counterexample of splitting: main in the introduction}). Then we study the $k$-diameter of closed manifolds.

\subsection{Acknowledgements}
 The first author would like to thank his supervisor Prof Haizhong Li for his encouragement and support. The second author wants to thank Professors Gaoming Wang and Zhifei Zhu for discussions on the splitting theorem. The authors also appreicate Professor Man-Chun Lee for discussions. The second author is supported by NSFC No. 12401058 and the Talent
Fund of Beijing Jiaotong University No. 2024XKRC008.

\section{Reduction of spectral $m$-intermediate curvatures}\label{section2}

\begin{definition}
We say that $N$ has nonnegative (respectively, positive; uniformly positive) spectral $(k,m)$-intermediate curvature if there exist a positive function $u \in C^{2,\beta}(N)$ for some $\beta \in (0,1)$ and a constant $k \geq 0$ such that
\[
 -k \Delta_N u + C^N_m\, u \geq 0 
 \quad (\text{respectively, } > 0;\ \geq \lambda u \text{ for some } \lambda > 0).
\]
In particular, when $k = 0$, this reduces to pointwise nonnegative (respectively, positive; uniformly positive) $m$-intermediate curvature.
\end{definition}

Note that whenever we say $N$ ha nonnegative spectral $(k,m)$-intermediate curvature, there is a corresponding positive function $u$ on $N$.

Two particular cases are previously well studied.
\begin{definition}
    We say that $N$ has nonnegative (respectively, positive; uniformly positive) spectral $k$-scalar curvature if there exist a positive function $u \in C^{2,\beta}(N)$ for some $\beta \in (0,1)$ and a constant $k \geq 0$ such that
    \[-k \Delta_N u+R_N u \geq 0 \quad (\text{respectively, } > 0; \ \geq \lambda u \text{ for some } \lambda > 0).\]
\end{definition}

\begin{definition}
    We say  that $N$ has nonnegative (respectively, positive; uniformly positive) spectral $k$-Ricci curvature if there exist a positive function $u \in C^{2,\beta}(N)$ for some $\beta \in (0,1)$ and a constant $k \geq 0$ such that
    \[-k \Delta_N u+\Ric_N u \geq 0 \quad (\text{respectively, } > 0;\ \geq \lambda u \text{ for some } \lambda > 0).\]
\end{definition}

For the spectral $k$-Ricci curvature, there are two important results.
\begin{theorem}\cite{shenyingyerugang,antonelli-xu}\label{Diameter estimate of Shen-ye}
Let $N^n$ be $n$-dimensional complete manifold, and let $0\leq \gamma<\frac{4}{n-1}$ when $n>3$, or $0\leq \gamma\leq 2$ when $n=3$, $\lambda>0$. Assume there exists a positive function $u$ such that 
\[\gamma\Delta_N u\leq u\Ric_N-(n-1)\lambda u.\]
Then
\[\operatorname{diam}(N)\leq \sqrt{n-1+\frac{(n-3)^2}{\frac{4}{\gamma}-n+1}}\frac{\pi}{\sqrt{(n-1)\lambda}}.\]
Moreover, the inequality is strict unless $n=3.$
\end{theorem}
\begin{theorem}[\cite{antonelli-xu}]\label{diameter bound under spectral Ric}
    Let $N^n$ be an $n$-dimensional closed smooth Riemannian manifold with $n\geq 3$, and let $0\leq \gamma \leq \frac{n-1}{n-2}$. Assume there is a positive function $u$ and $\lambda>0$ such that
    \[\gamma\Delta_N u\leq u\Ric_N-(n-1)\lambda u.\]
    Then 
    \[\operatorname{diam}(N)\leq \frac{\pi}{\sqrt{\lambda}}\cdot \left(\frac{u_{\max}}{u_{\min}}\right)^{\frac{n-3}{n-1}\gamma}.\]
    In particular, $N$ has finite fundamental group.
\end{theorem}

\begin{definition}
Let $u$ be a positive function on $N$ and $k\geq 0$ a real number, we say the oriented, embedded hypersurface $\Sigma$ is (absolutely) $\mathcal{A}_k(u)$-minimizing if for any open set $\Omega\subset N$ with compact closure, we have
\[\int_{\Omega\cap \Sigma}u^k\leq \int_{\Omega\cap S}u^k\]
for any oriented, embedded competitor $S$ satisfying $\partial S=\partial\Sigma$ and $S\setminus \Omega=\Sigma\setminus \Omega.$
\end{definition}

\vskip.2cm
For integers $n\geq 3 $ and $1\leq m\leq n-1$ such that $m^2-mn+2n-2>0$ and $m^2-mn+m+n>0$, we denote 
\[D(n,m)=\min\left\{\frac{m}{2m-2},\frac{1}{n-m},\frac{m^2-mn+m+n}{2(m^2-mn+2n-2)}\right\}.\]
We prove

\begin{theorem}\label{spectral inheritance}
    Let $N$ be a $n$-dimensional Riemannian manifold with positive spectral $(k,m)$-intermediate curvature for $k<4$. Suppose $\Sigma$ is an $(n-1)$-dimensional hypersurface that is the minimizer of $\mA_{k}=\int_\Sigma u^k$. Assume that $D(n,m)\geq \frac{k-1}{k}$, then
    
        (1) If $\Sigma$ is closed,  $\Sigma$ admits positive spectral $(\frac{4}{4-k},m-1)$-intermediate curvature.
        
        (2) If $\Sigma$ is complete noncompact, $\Sigma $ admits nonnegative spectral $(\frac{4}{4-k},m-1)$-intermediate curvature.

\end{theorem}
\begin{remark}
    In case (1), if $N$ has nonnegative (uniformly positive) spectral $(k,m)$-intermediate curvature, $\Sigma$ also admits nonnegative (uniformly positive) spectral $(\frac{4}{4-k},m-1)$-intermediate curvature.
\end{remark}
\begin{remark}
    If $k=0$, $D(n,m)\geq \frac{k-1}{k}$ automatically holds since right-hand side is $-\infty$.
\end{remark}

\begin{proof}
    Since $\Sigma$ is the minimizer of $\mA_{k}$, we calculate the first and second variations of $\mA_{k}.$ The first variation is
  
       \begin{equation*}
        \frac{d}{d t}\bigg|_{t=0}\mA_{k}(\Sigma_t)=\int_{\Sigma_t}\left(H_\Sigma u^k+\lp\nabla^{N}u^k,\nu\rp\right)\psi d\mH^{n-1}
    \end{equation*}
    Thus $H_\Sigma=-k \nabla^N_\nu\log u$ on $\Sigma$. 
    The second variation is \begin{equation*}
        \begin{split}
            \frac{d^2}{d t^2}\bigg|_{t=0}\mA_{k}(\Sigma_t)&=\int_{\Sigma}|\nabla^{\Sigma}\psi|^2 u^k-\left(|A_{\Sigma}|^2+\Ric_N(\nu,\nu)\right)\psi^2 u^k\\
            &+\int_\Sigma k(k-1)u^{k-2}\left(\nabla^N_{\nu}u\right)^2\psi^2 +\int_{\Sigma}ku^{k-1}\left(\Delta_N u-\Delta_{\Sigma} u\right)\psi^2\\
            &\geq 0.
        \end{split}
    \end{equation*}
    Let $\psi=u^{-k/2}\phi$. 
    Then
    \[|\nabla^{\Sigma}\psi|^2 =u^{-k}|\nabla^\Sigma\phi|^2+\frac{k^2}{4}\phi^2u^{-k-2}|\nabla^\Sigma u|^2-ku^{-k-1}\phi\nabla^\Sigma\phi\cdot\nabla^\Sigma u.\]
    When $k\neq 0$, we obtain
    \begin{equation}\label{A_k-stability inequality}
    \begin{split}
        0&\leq \int_\Sigma |\nabla^\Sigma \phi|^2+(ku^{-1}\Delta_N u-|A_\Sigma|^2-\Ric_N(\nu,\nu)+\frac{k-1}{k}H_{\Sigma}^2)\phi^2\\
        & +\int_\Sigma (\frac{k^2}{4}-k)\phi^2|\nabla^\Sigma \log u|^2+k\phi\nabla^\Sigma\phi\cdot\nabla^\Sigma\log u
    \end{split}
    \end{equation}
    where we have used the critical equation. Using the Cauchy-Schwarz inequality
    \[k|\phi \nabla^\Sigma\phi||\nabla^\Sigma\log u|\leq \epsilon \phi^2|\nabla^\Sigma \log u|^2+\frac{k^2}{4\epsilon}|\nabla^\Sigma \phi|^2\]
    and taking $\epsilon=k-k^2/4$, we have
    \[0\leq \int_\Sigma \frac{4}{4-k}|\nabla^\Sigma\phi|^2+(ku^{-1}\Delta_N u-|A_\Sigma|^2-\Ric_N(\nu,\nu)+\frac{k-1}{k}H_{\Sigma}^2)\phi^2.\]
    
    It follows from the assumption that $M$ has positive spectral $(k,m)$-intermediate curvature that 
    \begin{equation*}
        0<\int_\Sigma \frac{4}{4-k}|\nabla^\Sigma\phi|^2+(C_m^N-|A_\Sigma|^2-\Ric_N(\nu,\nu)+\frac{k-1}{k}H_{\Sigma}^2)\phi^2.
    \end{equation*}
    Let $\nu$ be the unit normal vector field of $\Sigma$ in $N$. At a point $p\in \Sigma$, we assume that $\{e_2,\cdots,e_n\}$ is an  orthonormal basis of $T_p\Sigma$ such that $\{e_1=\nu,e_2,\cdots,e_n\}$ is an orthonormal basis of $T_pN$ and $C_{m-1}^\Sigma=C_{m-1}^\Sigma(e_2,\cdots,e_m).$ Then Gauss equation yields
\begin{align*}
    C_m^M&\leq \sum_{i=2}^{m}\sum_{j=i+1}^n R^M_{ijij}+\Ric_N(\nu,\nu)\\
    &=\sum_{i=2}^{m}\sum_{j=i+1}^n (R^\Sigma_{ijij}-h_{ii}h_{jj}+h_{ij}^2)+\Ric_N(\nu,\nu)\\ &=C^\Sigma_{m-1}-\sum_{i=2}^{m}\sum_{j=i+1}^n (h_{ii}h_{jj}-h_{ij}^2)+\Ric_N(\nu,\nu)
\end{align*}    
where $h(\cdot,\cdot)$ is the second fundamental form of $\Sigma$ in $N$ with respect to the unit normal $\nu$. Now we have
\[0< \int_\Sigma \frac{4}{4-k}|\nabla^\Sigma\phi|^2+\left(C_{m-1}^\Sigma-|A_\Sigma|^2-\sum_{i=2}^{m}\sum_{j=i+1}^n (h_{ii}h_{jj}-h_{ij}^2)+\frac{k-1}{k}H_{\Sigma}^2\right)\phi^2.\]

In the following we directly cite a sharp algebraic inequality proved by Chen \cite[Lemma 5.8]{chenshuli_end}:
\[|A_\Sigma|^2+\sum_{i=2}^{m}\sum_{j=i+1}^n (h_{ii}h_{jj}-h_{ij}^2)\geq D(n,m)H_\Sigma^2\]
if $m^2-mn+2n-2>0$ and $m^2-mn+m+n>0$.

Using above inequality, we have
\[0<\int_\Sigma \frac{4}{4-k}|\nabla^\Sigma\phi|^2+C_{m-1}^\Sigma\phi^2\]
When $\Sigma$ is compact, it implies that there exists a positive function $v$ on $\Sigma$ such that
\[-\frac{4}{4-k}\Delta_\Sigma v+C^\Sigma_{m-1}v=\lambda_1 u> 0.\]
When $\Sigma$ is noncompact it follows from Fischer-Colbrie and Schoen \cite{Fischer-Colbrie-Schoen-The-structure-of-complete-stable} that there exists a positive function $v$ on $\Sigma$ such that 
\[-\frac{4}{4-k}\Delta_\Sigma v+C^\Sigma_{m-1}v=0.\]

\end{proof}

\begin{remark}\label{brendle's inequality}
    When $\Sigma$ is minimal, Brendle-Hirsch-Johne proved \cite{brendlegeroch'sconjecture} that
\[|A_\Sigma|^2+\sum_{i=2}^{m}\sum_{j=i+1}^n (h_{ii}h_{jj}-h_{ij}^2)\geq 0   \] 
if $m^2-mn+2n-2>0$. Equality holds iff $\Sigma$ is totally geodesic.
\end{remark}

Note that for $n\geq 3$, we always have $D(n,n-1)\geq \frac{1}{2}+\frac{1}{2n-2}>\frac{1}{2}$. Thus
\begin{corollary}
    Let $N$ be an $n$-dimensional Riemannian manifold with positive spectral $k$-scalar curvature. Assume $k<2$. Suppose $\Sigma$ is a $(n-1)$-dimensional hypersurface that is the smooth minimizer of $\mA_{k}=\int_\Sigma u^k$. Then $\Sigma$ admits positive spectral $\frac{4}{4-k}$-scalar curvature.
\end{corollary}

This inheritance of positive spectral scalar curvature resembles the fact that stable minimal hypersurface admits pointwise positive scalar curvature.

\vskip.2cm
Note that  for $3\leq n\leq 7$ and $1\leq m\leq n-1$ such that $m^2-mn+2n-2>0$ and $m^2-mn+m+n>0$, we always have
\[D(n,m)=\frac{m^2-mn+m+n}{2(m^2-mn+2n-2)}.\]
To make sure we can iterate Theorem \ref{spectral inheritance} by $m-1$ times from $m$-intermediate curvature to spectral $(0,1)$-intermediate (Ricci) curvature, we need to make sure
\begin{equation}\label{D-inequality}
    D(n-\ell,m-\ell)\geq \frac{\ell-1}{2\ell}
\end{equation}
for $0\leq \ell\leq m-2.$ This is true by case-checking.

Thus we can give a proof of the result of Brendle-Hirsch-Johne \cite{brendlegeroch'sconjecture}.
\begin{theorem}[\cite{brendlegeroch'sconjecture}]\label{a new proof of brendle's result}
    The compact product manifold \( N^n = M^{n-m} \times \mathbb{T}^m \) do not admit a metric of positive \( m \)-intermediate curvature for \(3\leq  n \leq 7 \) and \( 1 \leq m \leq n - 1 \).
\end{theorem}

\begin{proof}
    The dimension $n\leq 7$ is to ensure the regularity. We only need to prove the theorem when $m\geq 2$ due to Bonnet-Meyer's theorem. Assume that $N$ has positive $(0,m)$-intermediate curvature, i.e., $C_m^N>0.$ First we find an area minimizer $\Sigma^{n-1}$ (closed manifold) in the homology class $[M^{n-m}\times \mathbb{T}^{m-1}\times \{p\}]$. By Theorem \ref{spectral inheritance}, $\Sigma^{n-1}$ admits positive spectral $(1,m-1)$-intermediate curvature. We then continue to find minimizers of $\mathcal{A}_k$ and apply Theorem \ref{spectral inheritance} by $m-2$ more times. In the end, we can construct a $(n-m+1)$-dimensional closed manifold $\Sigma^{n-m+1}$ that maps to $M^{n-m}\times \mathbb{S}^1$ with nonzero degree and admits positive spectral $\frac{2m-2}{m}$-Ricci curvature.  It is not hard to check that 
    \begin{equation}\label{condition 1}
        \frac{2m-2}{m}\leq \frac{n-m}{n-m-1}
    \end{equation}
    holds for \(3\leq  n \leq 7 \) and \( 2 \leq m \leq n - 1 \). Then, according to Theorem \ref{diameter bound under spectral Ric}, $\Sigma^{n-m+1}$ has diameter upper bound. This is a contradiction, since its $\mathbb{R}$-cover satisfies same spectral condition but has infinite diameter. 
\end{proof}

\section{Splitting theorem}\label{section3}
In this section, we work on manifolds of type $M^{n-m}\times \mathbb{T}^{m-1}\times\mathbb{R}$ with a Riemannian metric $g$ where $M$ is some compact manifold.

A sharp spectral version of Cheeger-Gromoll splitting theorem was proved by Antonelli-Pozzetta-Xu and Catino-Mari-Mastrolia-Roncoroni.
\begin{theorem}[\cite{antonelli-pozzetta-xu,catino-Mari-Mastrolia-Roncoroni}]\label{splitting of nonnegative spectral ric}
    Let $n\geq 2,$ $\gamma<\frac{4}{n-1}$, and let $(N^n,g)$ be an $n$-dimensional smooth complete noncompact Riemannian manifold without boundary. Assume that $M$ has at least two ends and satisfies
    \[\lambda_1(-\gamma\Delta_N+\Ric_N)\geq 0.\]
    Then $\Ric\geq 0$ on $N$. In particular, the manifold splits off an $\mathbb{R}$ and a closed manifold $E$ with nonnegative Ricci curvature.
\end{theorem}

Antonelli-Pozzetta-Xu made use of $\mu$-bubble while Catino-Mari-Mastrolia-Roncoroni applied the criticality theory of Schrodinger operator. The assumption of two ends implies the existence of a geodesic line but assuming a geodesic line is not sufficient to prove splitting. In fact,  by perturbing the Euclidean metric in a small compact set, they can show in \cite{antonelli-pozzetta-xu} that $\mathbb{R}^n$ with this new metric contains a geodesic line ( infinitely many) and has nonnegative $\gamma$-spectral Ricci curvature. Apparently, this manifold does not split. By using a different idea, the second author with Wang \cite{hong-wang-splitting} showed that if $M$ has nonnegative $\gamma$-spectral Ricci curvature and
contains a weighted geodesic line (exists if $M$ has at least two ends), i.e., an $\mathbb{R}$-minimizer to the functional $\int u^\gamma$ where $-\gamma \Delta_N u+\Ric_N u=0$ and $\gamma<\frac{4}{n-1}$, then $N$ splits.
\vskip.1cm
The main theorem of this section is the following (Theorem \ref{splitting of nonnegative C_m: main in the introduction} in the introduction).
\begin{theorem}\label{splitting of nonnegative C_m}
Assume either \(3 \leq  n \leq  5\), \(1 \leq  m \leq  n-1\) or \(6 \leq  n \leq  7\), \(m\in \{1,n-2,n-1\}\). Let \((N^n,g)\) be an orientable complete noncompact Riemannian manifold with nonnegative $m$-intermediate curvature, which admits a smooth proper map \(f : N \to M^{n-m}\times \mathbb{T}^{m-1} \times \mathbb{R}\) with nonzero degree. Then \((N,g)\) is isometric to the Riemannian product of a closed manifold $E$, flat $(m-1)$-torus and the real line where $E$ has nonnegative Ricci curvature.
\end{theorem}

This theorem implies that $N=(M^{n-m}\times \mathbb{T}^m)\# X$ for an arbitrary complete manifold $X$ with nonnegative $m$-intermediate curvature must split. This is because a cyclic cover of $N$ admits a smooth proper map to $M^{n-m}\times \mathbb{T}^{m-1}\times \mathbb{R}$ with nonzero degree. It also implies that in the given range of $m,n$, such topological manifolds can not have positive $m$-intermediate curvature.

 We first prove the following result which is inspired by the works of Liu \cite{Liu-nonnegative-Ricci-curvature}, Chodosh-Eichmair-Moraru \cite{chodoshsplittingtheorem} and Hong-Wang \cite{hong-wang-splitting}.

\vskip.4cm
\begin{theorem}\label{spectral splitting theorem of C_2}
    Let $(N^{\ell+2},g)$ be an orientable complete noncompact Riemannian manifold with nonnegative spectral $(k,2)$-intermediate curvature, i.e., there exists a positive  function $u\in C^{2,\beta}(N)$ satisfying
    \[-k\Delta_N u+C_2^N u\geq 0.\] Assume that $M$ is an $\ell$-dimensional closed manifold and $N$ admits a smooth proper map $f:N\rightarrow M\times \mathbb{S}^1\times\mathbb{R}$ with nonzero degree and $\ell\geq1,k\geq 0, \ell+k<4$. Then up to scaling, $N$ is isometric to $E\times\mathbb{S}^1\times\mathbb{R}$ with the standard metric $g_E+d\theta^2+dt^2$ where $E$ is a closed manifold with nonnegative Ricci curvature.
\end{theorem}
\begin{proof}
We split the proof into several claims. Let $d\theta$ be the normalized one-form on $\mathbb{S}^1$. 
\vskip.2cm
\textbf{Claim 1:} $N$ contains a complete noncompact $\mathcal{A}_k(u)$-minimizing hypersurface $\Sigma$ such that there exists a smooth proper map from $\Sigma$ to $M\times\mathbb{R}$ with nonzero degree.
\vskip.1cm
    First, it is readily to see that $N$ has at least two ends due to properness of the map. Denote 
    $$\pi_1:M\times \mathbb{S}^1\times \mathbb{R}\rightarrow M\times \mathbb{S}^1\ \ \text{and}\ \ \pi_2:M\times \mathbb{S}^1\times \mathbb{R}\rightarrow  \mathbb{R}.$$ 
    
    For generic value $0\geq h\in \mathbb{R}$, $\Gamma_{ h}:=(\pi_2\circ f)^{-1}( h)$ and $\Gamma_{- h}:=(\pi_2\circ f)^{-1}(-h)$ are closed embedded separable hypersurface in $M$
    with $(\pi_1\circ f)_*([\Gamma_{\pm h}])=(\text{deg} f)[M\times \mathbb{S}^1]\in H_{\ell}(M\times \mathbb{S}^{1})$. We can pick up some connected component of $\Gamma_h$, still denoted by $\Gamma_h$, such that 
    \begin{equation}\label{topology of the boundary}
        (\pi_1\circ f)_*([\Gamma_h])=c[M\times \mathbb{S}^1], \ \ \ c\neq 0.
    \end{equation}
    We do the same thing to $\Gamma_{-h}$.  For convenience, we denote $\tilde{\pi}:M\times \mathbb{S}^{1}\times \mathbb{R}\rightarrow M\times \mathbb{R}$. This is an element $[\tilde{\Gamma}_h]=[\Gamma_{h}]\frown [f^{*}d\theta]\in H_{\ell}(N)$ satisfying
    \begin{equation}\label{topology of the boundary 1}
         (\pi_{1}\circ f)_{*}([\tilde{\Gamma}_{h}])=c[M]\in H_{\ell}(M\times \mathbb{S}^{1};\mathbb{Z}).
    \end{equation}
    Moreover, $(\tilde{\pi}\circ \pi_{1}\circ f)_{*}([\tilde{\Gamma}_{h}])=c[M]\in H_{\ell}(M,\mathbb{Z})$. In particular, $[\tilde\Gamma_{h}]\neq 0\in H_{\ell}(N,\mathbb{Z})$ and $(\tilde{\pi}\circ f)_{*}([\tilde{\Gamma}_{h}])\neq 0\in H_{\ell}(M\times \mathbb{R})$. We can also find $\tilde{\Gamma}_{-h}.$
    
     
     Let $\Omega_h$ be a compact set of $N$ that contains $\Gamma_0$ and $\Gamma_{\pm h}$. We would like to minimize the functional $\mathcal{A}_k(u)$ among all compact hypersurfaces whose boundaries are $\Gamma_{\pm h}$. We can perturb the metric near $\partial\Omega_h$ such that the mean curvature is positive with respect to the outer normal vector and the metric $\tilde{g}=u^{2k/\ell+1}g.$ Thus there exists a compact smooth, two-sided, embedded $\mathcal{A}_k(u)$-minimizing hypersurface for each $h>0$, which we call $\Sigma_h$. We claim that $\Sigma_h$ is connected. Indeed, if not, then there exists $(\ell+1)$-cell spanning $\Gamma_h$ or $\Gamma_{-h}$, which is a contradiction to \eqref{topology of the boundary 1}. On the other hand, since $\Gamma_0$ separates $\Gamma_h$ and $\Gamma_{-h}$, we know that for any $h>0$, $\Sigma_h$ must intersect $\Gamma_0$.

     For a real number $a>0$, denote $K_{a}=f^{-1}(M\times \mathbb{S}^{1}\times [-a,a])$.
     It follows from  the connectedness of $\Sigma_{h}$ that for $a\in[-h,h]$, we have 
     \begin{equation}\label{slicing has same homology}
         [\Sigma_{h}\cap (\pi_{2}\circ f)^{-1}(a)]=[\tilde{\Gamma}_{h}]\in H_{\ell}(\Sigma_{h},\mathbb{Z}).
     \end{equation}
     Now we claim that the map $$\tilde{\pi}\circ f:\Sigma_{h}\cap K_{a}\rightarrow M\times \mathbb{S}^{1}\times [-a,a]\rightarrow M\times [-a,a]$$ has nonzero degree, that is, $$(\tilde{\pi}\circ f)_{*}([\Sigma_{h}\cap K_{a}])=\tilde{c}[M\times [-a,a]]\in H_{\ell+1}(M\times [-a,a],M\times \{-a,a\}),\ \tilde{c}\neq 0.$$ On the contrary suppose $\tilde{c}=0$, then $$\partial ((\tilde{\pi}\circ f)_{*}([\Sigma_{h}\cap K_{a}]))=(\tilde{\pi}\circ f)_{*}([\Sigma_{h}\cap (\pi_{2}\circ f)^{-1}(a)])-(\tilde{\pi}\circ f)_{*}([\Sigma_{h}\cap (\pi_{2}\circ f)^{-1}(-a)])$$
     is trivial. Therefore, $(\tilde{\pi}\circ f)_{*}([\Sigma_{h}\cap (\pi_{2}\circ f)^{-1}(a)])=0\in H_{\ell}(M\times \{a\},\mathbb{Z})$. In particular, this yields $(\tilde{\pi}\circ f)_{*}([\Sigma_{h}\cap (\pi_{2}\circ f)^{-1}(a)])=0\in H_{\ell}(M\times \mathbb{R},\mathbb{Z})$ which is a contradiction to \eqref{topology of the boundary 1} and \eqref{slicing has same homology}. 
     
  Using curvature estimates, a subsequence of $\Sigma_h$ converges to an oriented, two-sided, embedded hypersurface $\Sigma$ that is absolutely $\mathcal{A}_k(u)$-minimizing. Moreover, for any fixed $ a>0$, $\Sigma_{h}\cap K_{a}$ graphically and smoothly converge to $\Sigma\cap K_{a}$ as $h\rightarrow \infty$. In particular, there are no closed components in $\Sigma.$ Combining with fact that the map $\tilde{\pi}\circ f:\Sigma_{h}\cap K_{a}\rightarrow M\times [-a,a]$ has nonzero degree, we obtain $(\tilde{\pi}\circ f)|_{K_{a}}:\Sigma\cap K_{a}\rightarrow M\times [-a,a]$ has nonzero degree. Taking $a$ to infinity, we obtain $\tilde{\pi}\circ f: \Sigma\rightarrow M\times \mathbb{R}$ has nonzero degree. 
  Note that $\Sigma$ need not be connected, since the functional $\mathcal{A}_k(u)$ can become arbitrarily small along certain ends of $N$, allowing the limit of hypersurface to break. However, we can take one component of $\Sigma$ which admits a smooth proper map into $M \times \mathbb{R}$ of nonzero degree.

\vskip.2cm
\textbf{Claim 2:} $\Sigma$ is isometric to a product manifold $(E\times \mathbb{R},g_E+dt^2)$ where $E$ is closed $\ell$-manifold and has nonnegative Ricci curvature. 
\vskip.1cm
    Indeed, since $\Sigma$ is $\mathcal{A}_k(u)$-minimizing, according to Theorem \ref{spectral inheritance}, $\Sigma$ admits nonnegative spectral $(\frac{4}{4-k},1)$-intermediate curvature. That is, there exists a positive function $v$ on $\Sigma$ such that
    \[-\frac{4}{4-k}\Delta_\Sigma v+\Ric_\Sigma v=0.\]
    Note the dimension of $\Sigma$ is $\ell+1$. Then the claim follows from Theorem \ref{splitting of nonnegative spectral ric} and the assumption $\ell+k<4$.

\vskip.2cm
    \textbf{Claim 3:} $\Sigma$ is totally geodesic. Moreover, $|\nabla^N u|=0$ and $-k\Delta_N u+C_2^Nu= 0$ along $\Sigma$.
\vskip.1cm

By \textbf{Claim 2} we know $\Ric_\Sigma(\partial_t,\partial_t)=0$ and $\Sigma$ has linear volume growth. For every $r>0$ let $\psi_r$ be a cutoff function on $\mathbb{R}$ such that $\psi_r\equiv 1$ on $[-r,r]$, $\psi_r\equiv 0$ on $\mathbb{R}\setminus [-2r,2r]$, and linear on  $[-2r,-r]\cup [r,2r]$. Let $\phi_r(x,t):=\psi_r(t)$ be defined on $E\times \mathbb{R}$. Let $\partial_t$ be the unit tangent vector in the $\mathbb{R}$-direction. Plugging $\phi_r$ into the $\mathcal{A}_k(u)$-stability inequality,
\begin{equation*}
    \begin{split}
        \int_\Sigma |\nabla^\Sigma \phi_r|^2&\geq \int_\Sigma (|A_\Sigma|^2+\Ric_N(\nu,\nu)-ku^{-1}\Delta_N u-\frac{k-1}{k}H_{\Sigma}^2)\phi_r^2\\
        & +\int_\Sigma (k-\frac{k^2}{4})\phi_r^2|\nabla^\Sigma \log u|^2-k\phi_r\nabla^\Sigma\phi_r\cdot\nabla^\Sigma\log u.
    \end{split}
    \end{equation*}
    By Cauchy-Schwarz inequality,
\begin{equation}\label{rigidity inequality}
    \begin{split}
        \frac{4}{4-k}\int_\Sigma |\nabla^\Sigma \phi_r|^2&\geq \int_\Sigma (|A_\Sigma|^2+\Ric_N(\nu,\nu)-ku^{-1}\Delta_N u-\frac{k-1}{k}H_{\Sigma}^2)\phi_r^2\\
        &\geq \int_\Sigma (|A_\Sigma|^2+\Ric_N(\nu,\nu)-C_2^N u-\frac{k-1}{k}H_{\Sigma}^2)\phi_r^2\\
        &\geq \int_\Sigma (|A_\Sigma|^2+\Ric_N(\nu,\nu)-C_2^N(\partial_t,\nu) u-\frac{k-1}{k}H_{\Sigma}^2)\phi_r^2.
    \end{split}
    \end{equation}
    Let $\{e_1=\nu,e_2=\partial_t,e_3,\cdots,e_{\ell+2}\}$ be a local orthonormal frame. By Gauss equation, $$0=\Ric_\Sigma(\partial_t,\partial_t)=\Ric_N(\partial_t,\partial_t)-R^N(\nu,\partial_t,\nu,\partial_t)+\sum_{i=2}^{\ell+2}(h_{ii}h_{22}-h_{i2}^2)$$
    and
    $$C_2^N(\partial_t,\nu)=\Ric_N(\nu,\nu)+\Ric_N(\partial_t,\partial_t)-R^N(\nu,\partial_t,\nu,\partial_t).$$
    Thus
    \begin{align*}
         \frac{4}{4-k}\int_\Sigma |\nabla^\Sigma \phi_r|^2&\geq \int_\Sigma (|A_\Sigma|^2+\sum_{i=2}^{\ell+2}(h_{ii}h_{22}-h_{i2}^2)-\frac{k-1}{k}H_{\Sigma}^2)\phi_r^2\\
         &\geq \int_\Sigma (D(\ell+2,2)-\frac{k-1}{k})H_{\Sigma}^2\phi_r^2\\
         & \geq \int_\Sigma (\frac{4-\ell}{4}-\frac{k-1}{k})H_{\Sigma}^2\phi_r^2.
    \end{align*}
   Letting $r \rightarrow \infty$, we obtain that $\Sigma$ is totally geodesic, $\nabla^N_\nu u= 0$, $\nabla^\Sigma u= 0$ along $\Sigma$. Moreover, $-k\Delta_N u+C_2^N(\partial_t,\nu)u= 0$ along $\Sigma$.   
    \vskip.2cm

The remaining goal is to show that for any point $p\in N$, we have
\[\Ric_N\geq 0.\]
    By \textbf{Claim 3}, $\Sigma$ splits to $E\times \mathbb{R}$ with product metric. We cut $\Sigma$ from $N$ to get a new $(\ell+2)$-manifold $\hat{N}$ with two boundary components both isometric to $\Sigma$ and with the metric $\hat{g}$.  We fix a boundary component $\Sigma$. Let $h > 0$ be a real number. Denote by $\Sigma_h = E \times [-h, h]$ the truncated part of $\Sigma$. Denote the closed manifold on $\Sigma$ that corresponds to $E\times \{0\}$ by $E_0$. Let $\nu$ be the inner unit normal vector field of $\Sigma$ in $\hat{N}.$
    
    Fix a unit speed geodesic $c:[0,\varepsilon)\rightarrow \hat{N}$ with $c(0)\in E_0$ and $\dot{c}(0)$ is perpendicular to $\Sigma$. By standard perturbation, we can construct a family of positive functions \( u_{r,t} \) (see [HW25, Lemma 3.3] for the construction) with the following properties:
\begin{enumerate}[(a)]
\item \( u_{r,t} \to u \) in \( C^3 \) as \( t, r \to 0 \);
\item \( u_{r,t} \to u \) smoothly as \( t \to 0 \) for \( r \in (0, \varepsilon) \) fixed;
\item \( u_{r,t} = u \) on \( \{ x \in \hat{N} : \operatorname{dist}_{\hat{g}}(x, c(2r)) \ge 3r \} \);
\item \( u_{r,t} < u \) on \( \{ x \in \hat{N} : \operatorname{dist}_{\hat{g}}(x, c(2r)) < 3r \} \);
\item \( -k\Delta_{\hat{g}} u_{r,t} +  C_2^N u_{r,t} > 0 \) on \( \{ x \in \hat{N} : r < \operatorname{dist}_{\hat{g}}(x, c(2r)) < 3r \} \);
\item \( H_\Sigma+\partial_\nu u_{r,t} \ge 0 \).
\end{enumerate}
We further construct new smooth positive functions \( u_{r,t,h} \) for any \( h > 0 \) by requiring

\begin{enumerate}[(i)]
\item \( u_{r,t,h} = u_{r,t} \) in \( \{ x \in \hat{N} : \operatorname{dist}_g(x, \Sigma_{h+1})) \leq 2h \} \),

\item \( u_{r,t,h} \geq u_{r,t} \) in \( \hat{N} \setminus \{ x \in \hat{N} : \operatorname{dist}_{\hat{g}}(x, \Sigma_{h+1}) \leq 2h \} \),

\item \( u_{r,t,h} \geq \max\{1, u_{r,t}\} \) in \( \{ x \in \hat{N} : \operatorname{dist}_g(x, \Sigma_{h+1}) \geq 2h +1 \}\).

\item \(H_\Sigma+\partial_\nu u_{r,t,h} \ge 0 \) in $\Sigma_{h+1}.$
\end{enumerate}

We consider all compact, oriented hypersurfaces in \( \{ x \in \hat{N} : \operatorname{dist}_g(x, \Sigma_{h+1})) \leq 2h \} \) that has the same boundary as $\Sigma_h$ and is homologous to $\Sigma_h$, there is a smooth oriented, two-sided, embedded $\mathcal{A}_k(u_{r,t,h})$-minimizer, which we denote by $\Sigma_{r,t,h}$. The condition (iv) ensures no interior touching and regularity of the minimizer.

\vskip.2cm
\textbf{Claim 4:} $\Sigma_{r,t,h}$ always intersects \( \{ x \in \hat{N} : \operatorname{dist}_{\hat{g}}(x, c(2r)) < 3r \} \) for any $h>0$.
\vskip.1cm
    This can be proved same as in \cite[equation (3.4)]{hong-wang-splitting}.
\vskip.2cm

 Similar to \textbf{Claim 1}, there exists a subsequence of $\Sigma_{r,t,h}$ converges locally and smoothly to a complete hypersurface $\tilde{\Sigma}_{r,t}$ as $h\rightarrow \infty$. This is absolutely
 $\mathcal{A}_k(u_{r,t})$-minimizing since $u_{r,t,h}\rightarrow u_{r,t}$ as $h\rightarrow \infty.$ Moreover, it admits nonzero degree map to $M\times \mathbb{R}$.  Since $\tilde{\Sigma}_{r,t}$ intersects with \( \{ x \in \hat{N} : \operatorname{dist}_{\hat{g}}(x, c(2r)) \leq  3r \} \), we denote $\Sigma_{r,t}$ the unique component of $\Sigma_{r,t,h}$ that intersects with \( \{ x \in \hat{N} : \operatorname{dist}_{\hat{g}}(x, c(2r)) \leq 3r \} \). Again $\Sigma_{r,t}$ is absolutely
 $\mathcal{A}_k(u_{r,t})$-minimizing and admits nonzero degree map to $M\times \mathbb{R}$. 
\vskip.2cm

\vskip.2cm
\textbf{Claim 5:} For any $r,t>0$ small enough, $\Sigma_{r,t}$ must either intersect \( \{ x \in \hat{N} : \operatorname{dist}_{\hat{g}}(x, c(2r)) \leq  2r \} \)  or intersect \( \{ x \in \hat{N} : \operatorname{dist}_{\hat{g}}(x, c(2r)) =  3r \} \) but not \( \{ x \in \hat{N} : \operatorname{dist}_{\hat{g}}(x, c(2r)) < 3r \} \). 
\vskip.1cm

If not, according to the construction of $u_{r,t}$, we always have $-k\Delta_{\hat{g}} u_{r,t}+C_2^N u_{r,t}\geq 0$ on $\hat{M}$. However, since $\Sigma_{r,t}$ is $\mathcal{A}_{k}(u_{r,t})$-minimizing, we obtain $-k\Delta_{\hat{g}} u_{r,t}+C_2^N u_{r,t}=0 $ along $\Sigma_{r,t}$ by \textbf{Claim 2}. This is contradiction to $(e)$ in the construction of $u_{r,t}$, thus proving the claim.

\vskip.2cm
 Fix $r>0$ small, as $t\rightarrow 0$, there is a subsequence of $\Sigma_{r,t}$ that converges to a embedded hypersurface, denoted by $\Sigma_r$, that intersects with \( \{ x \in \hat{N} : \operatorname{dist}_{\hat{g}}(x, c(2r)) \leq 3r \} \) and that is disjoint from $\Sigma$ by maximum principle. Also $\Sigma_r$ is absolutely $\mathcal{A}_k(u)$-minimizing. We know that for $r$ sufficiently small, $\Sigma_r$ admits a nonzero degree smooth proper map to $M\times \mathbb{R}$ (so it has at least two ends). By \textbf{Claim 2} and \textbf{Claim 3}, $\Sigma_r$ is product manifold, totally geodesic, and the normal Ricci curvature vanishes along $\Sigma_r$. Since $\Sigma_r$ converges to $\Sigma$ as $r\rightarrow 0$, by differentiating the second fundamental form we find that the positive variational speed function $f\in C^\infty(\Sigma)$ satisfies \begin{equation}\label{equation coming from sequence of totally geodesic surfaces}
0=\nabla_\Sigma^2f(X,X)+R^N(X,\nu,X,\nu)
 \end{equation}
 where $X$ is tangential to $\Sigma$.  Here $f$ is positive since the sequence of hypersurfaces converges from one side. In fact, one can write the sequence over $\Sigma$ by the graph function $f_i$, normalize $f_i$ at one point and apply Harnack estimate and Schauder estimates to get the limit function $f$. 

\vskip.2cm
\textbf{Claim 6:} $\Delta_\Sigma f=0$.
\vskip.1cm
This follows by taking the trace of equation \eqref{equation coming from sequence of totally geodesic surfaces} if we can show that $\Ric_N(\nu,\nu)$ along $\Sigma.$ From \text{Claim 3}, we know 
  \[|\nabla^N u|=0 \ \text{and}\ -k\Delta_N u+C_2^N u=0\]
  along the sequence of totally geodesic hypersurfaces $\Sigma_r$. For any point $x\in \Sigma$, consider an arclength parametrizied curve $\gamma(t):[0,\epsilon)\rightarrow N$ such that $\gamma(0)=x$ and $\gamma'(0)=\nu_x$. Let $t_r$ be  positive numbers  such that $\gamma(t_r)=\gamma(t)\cap\Sigma_r.$ We know
  \[u(\gamma(t))'|_{t=t_r}=0.\]
  Then
  \[u(\gamma(t))''|_{t=0}=0.\]
  That is
  \[((\nabla_N^2(\gamma(t)))(\dot{\gamma}(t),\dot{\gamma}(t))+\nabla_N u(\gamma(t))\cdot \ddot{\gamma}(t))|_{t=0}=0.\]
  Thus
  \[(\nabla_N^2u)(\nu,\nu)=0.\]
  It follows that
  \[\Delta_N u=(\nabla_N^2 u)(\nu,\nu)+\Delta_\Sigma u+ u_\nu H=0,\]
  so $C_2^N=0$ along $\Sigma$. By the proof of \textbf{Claim 2}, we know that $\Ric_N(\nu,\nu)=0.$ This completes the proof of the claim.

  \vskip.2cm

  \textbf{Claim 7:} For any point $p\in \Sigma,$ we have $\Ric_N(\cdot,\cdot)\geq 0$.
  \vskip.1cm
  Since $\Ric_N(\nu,\nu)\geq 0$ on $\Sigma$, it suffices to show $\Ric_N$ is nonnegative along tangential directions.
  Note that $\Sigma$ has linear volume growth. Applying integration by parts we obtain
 \begin{align*}
     \int_\Sigma \varphi^2|\nabla \log f|^2&=\int_\Sigma \varphi^2\Delta \log f\\&=-\int_\Sigma 2\varphi\nabla\varphi\nabla \log f\\&\leq \epsilon \int_\Sigma \varphi^2|\nabla\log f|^2+\frac{1}{\epsilon}\int_\Sigma |\nabla \varphi|^2
 \end{align*}
 and
 \[\int_\Sigma \varphi^2|\nabla \log f|^2\leq C\int_\Sigma |\nabla \varphi|^2.\]
 Choosing standard cutoff function $\varphi$ yields $|\nabla \log f|=0.$ Hence $f$ is constant function. 
 Then by \eqref{equation coming from sequence of totally geodesic surfaces}, $R^N(X,\nu,X,\nu)=0$ for any vector $X$ tangential to $\Sigma$. The claim holds.

\vskip.2cm
\textbf{Claim 8:} For any point $q\in N$, there is a $\mathcal{A}_k(u)$-minimizing cylinder passing through $p$. 
\vskip.1cm
Let $\Omega\subset M$ consist of all points in $N$ such that there is an $\mathcal{A}_k(u)$-minimizing cylinder passing through $p$. It is clear that $\Omega$ is nonempty closed subset by the compactness of $\mathcal{A}_k(u)$-minimizing cylinders. If $\Omega\neq N$, we can find any point $q\in N\setminus \Omega$. Let $q_0$ be the nearest point to $q$ in $\Omega$. Let $\Sigma$ be the cylinder passing through $q_0$. We can apply the above construction to obtain a new cylinder that is closer to $q$, contradicting the definition of $q_0.$ 
\vskip.2cm
Thus we have achieved the goal that $\Ric_N\geq 0$ by \textbf{Claim 7}, \textbf{Claim 8}, and then the theorem follows from Cheeger-Gromoll splitting theorem.
\end{proof}

\begin{remark}
    Recently, Chai-Sun \cite[Theorem 1.3]{chai-sun-spectralscalarsplitting} generalize the splitting theorem of Chodosh-Eichmair-Moraru \cite{chodoshsplittingtheorem} to the setting of spectral scalar curvature.
\end{remark}

\begin{remark}
    As mentioned in the introduction, it is still an interesting question to prove the following generalization of splitting theorem of Chodosh-Eichmair-Moraru \cite{chodoshsplittingtheorem}: Let $(N,g)$ be an orientable complete noncompact Riemannian manifold with nonnegative spectral $(k,2)$-intermediate curvature. Suppose that $N$ contains a properly embedded, absolutely $\mathcal{A}_k(u)$-minimizing hypersurface that has at least two ends. Then $N$ is isometrically covered by $E\times \mathbb{R}^2$.
\end{remark}

To extend above cylindrical splitting to higher dimensions, we need a rigidity result in higher dimensional case.
\begin{proposition}\label{rigidity of Cm in higher dimension}
   Fix integers $7\geq n>m>\ell\geq 0$. Let $N$ be $(n-\ell)$-dimensional Riemannian manifold of nonnegative spectral $(\frac{2\ell}{\ell+1},m-\ell)$-intermediate curvature. Assume that $\Sigma=(E^{n-m}\times \mathbb{T}^{m-\ell-2}\times \mathbb{R}, g_E+d\theta_1^2+\cdots+d\theta_{m-\ell-2}^2+dt^2)$, $E$ is closed, is a hypersurface in $N$ that minimizes $\mathcal{A}_{2\ell/(\ell+1)}(u)$. Then $\Sigma$ is totally geodesic. Moreover, $|\nabla^N u|=0$ and $-\frac{2\ell}{\ell+1}\Delta_N u+C_{m-\ell}^Nu= 0$ along $\Sigma$.
\end{proposition}
\begin{proof}
    Let $\partial_t$, $\partial_{\theta_i}$ for $i=1,\cdots,m-\ell-2$   be the unit tangent vectors in the $\mathbb{R}$-direction and $\mathbb{S}^1$-direction, respectively. Note that the sectional curvatures of $\Sigma$ satisfy $R^\Sigma(\partial_t,X,\partial_t,X)=0$and $R^\Sigma(\partial_{\theta_i},X,\partial_{\theta_i},X)=0$ for any tangential vector $X$ on $\Sigma$. For every $r>0$, let $\psi_r$ be a cutoff function on $\mathbb{R}$ such that $\psi_r\equiv 1$ on $[-r,r]$, $\psi_r\equiv 0$ on $\mathbb{R}\setminus [-2r,2r]$, and linear on  $[-2r,-r]\cup [r,2r]$. Let $\phi_r(x,\theta_1,\cdots,\theta_{m-\ell-2},t):=\psi_r(t)$, $x\in E$, be a cutoff function on $\Sigma$.   

\vskip.2cm
    \textbf{Case 1:} $\ell>0$.
    Plugging $\phi_r$ into the $\mathcal{A}_{2\ell/(\ell+1)}(u)$-stability inequality and using \eqref{rigidity inequality}, we have
    \begin{align*}
        \frac{2\ell+2}{\ell+2}\int_\Sigma |\nabla^\Sigma \phi_r|^2&\geq \int_\Sigma (|A_\Sigma|^2+\Ric_N(\nu,\nu)-\frac{2\ell}{\ell+1}u^{-1}\Delta_N u-\frac{\ell-1}{2\ell}H_{\Sigma}^2)\phi_r^2\\
        &\geq \int_\Sigma (|A_\Sigma|^2+\Ric_N(\nu,\nu)-C_{m-\ell}^N u-\frac{\ell-1}{2\ell}H_{\Sigma}^2)\phi_r^2\\
        &\geq \int_\Sigma (|A_\Sigma|^2+\Ric_N(\nu,\nu)-C_{m-\ell}^N(\partial_t,\partial_{\theta_1},\cdots,\partial_{\theta_{m-\ell-2}},\nu) u)\phi_r^2\\
        &-\int_\Sigma\frac{\ell-1}{2\ell}H_{\Sigma}^2\phi_r^2.
    \end{align*}
    Let $\{e_1=\nu,e_2=\partial_t,e_3=\partial_{\theta_1},\cdots,e_{m-\ell}=\partial_{\theta_{m-\ell-2}},e_{m-\ell+1},\cdots, e_{n-\ell}\}$ be a local orthonormal frame. 
    By Gauss equation, 
    \begin{align*}
        C_{m-\ell}^N(\partial_t,\partial_{\theta_1},\cdots,\partial_{\theta_{m-\ell-2}},\nu)&=\sum_{i=1}^{m-\ell}\sum_{j=i}^{n-\ell}R^N(e_i,e_j,e_i,e_j)\\
        &=\Ric_N(e_1,e_1)+\sum_{i=2}^{m-\ell}\sum_{j=i}^{n-\ell}(R^\Sigma_{ijij}-h_{jj}h_{ii}+h_{ji}^2)\\
        &=\Ric_N(e_1,e_1)-\sum_{i=2}^{m-\ell}\sum_{j=2}^{n-\ell}(h_{jj}h_{ii}-h_{ji}^2).
    \end{align*}
    Thus
    \begin{align*}
         \frac{2\ell+2}{\ell+2}\int_\Sigma |\nabla^\Sigma \phi_r|^2&\geq \int_\Sigma (|A_\Sigma|^2+\sum_{2}^{m-\ell}\sum_{j=2}^{n-\ell}(h_{ii}h_{jj}-h_{ij}^2)-\frac{\ell-2}{2\ell}H_{\Sigma}^2)\phi_r^2\\
         &\geq \int_\Sigma (D(n-\ell,m-\ell)-\frac{\ell-2}{2\ell})H_{\Sigma}^2\phi_r^2\\
         & \geq \int_\Sigma (\frac{\ell-1}{2\ell}-\frac{\ell-2}{2\ell})H_{\Sigma}^2\phi_r^2\\
         &=\int_\Sigma \frac{1}{2\ell}H_\Sigma^2\phi_r^2>0
    \end{align*}
   where in the third inequality we applied equation \eqref{D-inequality}. Letting $r \rightarrow \infty$, we obtain that $\Sigma$ is totally geodesic, $\nabla^N_\nu u= 0$, $\nabla^\Sigma u= 0$ along $\Sigma$. Moreover, $-\frac{2\ell}{\ell+1}\Delta_N u+C_{m-\ell}^Nu= 0$ along $\Sigma$. From these we know that $C_{m-\ell}^N=C_{m-\ell}^N(\partial_t,\partial_{\theta_1},\cdots,\partial_{\theta_{m-\ell-2}},\nu)=0$ on $\Sigma$.
\vskip.2cm
   \textbf{Case 2:} $\ell=0$. In this case, $\mathcal{A}_0(u)$ is just the area functional. Plugging $\phi_r$ into stability inequality, we have
   \begin{align*}
       \int_\Sigma|\nabla\phi_r|^2&\geq \int_\Sigma (\Ric_N(\nu,\nu)+|A_\Sigma|^2)\phi_r^2\\
       &\geq \int_\Sigma(C_m^N(e_1,\cdots,e_{m})+|A_\Sigma|^2+\sum_{i=2}^m\sum_{j=2}^{n}(h_{jj}h_{ii}-h_{ji}^2))\phi_r^2\\
       &\geq 0
   \end{align*}
   where in the last inequality we used Remark \ref{brendle's inequality}. Letting $r\rightarrow +\infty$, we obtain that $\Sigma$ is totally geodesic, $C_m^N=0$ and $\Ric_N(\nu,\nu)=0$ on $\Sigma$.
\end{proof}

We start to prove the main theorem.
\begin{proof}[Proof of Theorem \ref{splitting of nonnegative C_m}]
    The structure of the proof is illustrated in Figure \ref{picture of proof}. First, by similar arguments to \textbf{Claim 1} in the proof of Theorem \ref{spectral splitting theorem of C_2}. We can construct an $\mathcal{A}_0(1)$-minimizing (area-minimizing) hypersurface $\Sigma_1$ in $N$ that can be properly mapped to $M^{n-m}\times \mathbb{T}^{m-2}\times \mathbb{R}$ with nonzero degree. It follows from Theorem \ref{spectral inheritance} that $\Sigma_1$ admits nonnegative spectral $(1,m-1)$-intermediate curvature, i.e., there exists a positive function $u_1\in C^{2,\beta}$ for $\beta\in (0,1)$ such that $-\Delta_{\Sigma_1}u_1+C_{m-1}^{\Sigma_1} u_1=0.$ We repeat this procedure by $m-1$ times in total till the last slice $\Sigma_{m-1}$. Note that assumptions of Theorem \ref{spectral inheritance} is always satisfied in each step due to \eqref{D-inequality}. Eventually, we obtain a chain of manifolds $$\Sigma_{m-1}\subset \Sigma_{m-2}\subset \cdots\subset \Sigma_\ell\subset \cdots \subset \Sigma_1\subset N.$$ The following holds $$-\frac{2\ell}{\ell+1}\Delta_{\Sigma_{\ell}}u_{\ell}+C_1^{\Sigma_{\ell}}u_{\ell}=0$$ for a positive function $u_{\ell}$ on $\Sigma_{\ell}$.  In particular, the dimension of last slice $\Sigma_{m-1}$ is $n-m+1$ and it has at least two ends. One can check the following equivalence
    \[\frac{2m-2}{m}<\frac{4}{n-m}\ \ \Longleftrightarrow\ \ m^2-mn+m+n>0 \]
    which is satisfied by the assumption of $n,m$. This means that $\Sigma_{m-1}$ is a product manifold $E\times \mathbb{R}$ where $E$ is a closed manifold with nonnegative Ricci curvature.

    Since 
    $\frac{2m-4}{m-1}+n-m<4$, Theorem \ref{spectral splitting theorem of C_2} implies that up to scaling, $\Sigma_{m-2}$ is isometric to $E\times \mathbb{S}^1\times \mathbb{R}$  with canonical product metric.  
    
    In the following, we give a sketch of the splitting of $\Sigma_{m-3}$. After that, one can repeat the proof till getting the splitting of top manifold $N$.
    First, we know that there exists a smooth map from $\Sigma_{m-3}$ to $M\times \mathbb{T}^2\times \mathbb{R}$ with nonzero degree, $\Sigma_{m-3}$ has nonnegative spectral $(\frac{2m-6}{m-2},m-3)$-intermediate curvature, and $\Sigma_{m-2}$ is an $\mathcal{A}_{\frac{2m-6}{m-2}}(u_{m-3})$-minimizer. We can use the following steps (similar to the steps in the proof of Theorem \ref{spectral splitting theorem of C_2}) to prove $\Ric_{\Sigma_{m-3}}\geq 0$ everywhere and apply Cheeger-Gromoll splitting theorem to complete the proof.

    Step 1: Perturb the metric of $\Sigma_{m-3}$ near $\Sigma_{m-2}$ and construct compact minimizer $\Sigma^{r,t,h}_{m-2}$ in this new metric.

    Step 2: Show it always intersects $3r$-ball and then take limit to $\Sigma^{r,t}_{m-2}$ as $h\rightarrow \infty.$

    Step 3: Show $\Sigma^{r,t}_{m-2}$ 
    that can be mapped to $M\times \mathbb{S}^1\times \mathbb{R}$ with nonzero degree. 
    
    Step 4: Show $\Sigma^{r,t}_{m-2}$ always intersects $r$-ball. In this step we need to use Theorem \ref{spectral splitting theorem of C_2} and Proposition \ref{rigidity of Cm in higher dimension}. 

    Step 5: Show $\Ric_{\Sigma_{m-3}}\geq 0$ on $\Sigma_{m-2}$ and then extend it to everywhere in $\Sigma_{m-3}.$

    To conclude, we have proven the splitting of $N.$
\end{proof}


\section{Construction of examples}\label{section4}
We will use the following result frequently.
\begin{proposition}\label{lift C_m-1 to C_m}
Let $\Sigma$ be an $n$-dimensional Riemannina manifold and let $0\leq k <4$. If there exists a positive smooth function $u$ satisfies $-\frac{4}{4-k}\Delta_{\Sigma}u+C_{m-1}^{\Sigma}u\geq \lambda u$ for $\lambda\geq 0$. Then on the $(n+1)$-dimensional manifold $M=\Sigma\times \mathbb{S}^{1}$ with metric $g_{M}=g_{\Sigma}+u^{4\frac{2-k}{4-k}}d \theta^2$ we have 
\[
-k\Delta_{M}u^{\frac{2}{4-k}}+C_{m}^{M}(e_{1},...,e_{m-1},e_{\theta})u^{\frac{2}{4-k}} \geq \lambda u^{\frac{2}{4-k}}
\]
where by abuse of notation, $u$ is the constant extension  in the $\mathbb{S}^1$-direction, i.e., $u(\cdot,\theta)=u(\cdot)$, $e_\theta$ is the unit vector along  $\mathbb{S}^1$-direction, i.e., $e_\theta=u^{2\frac{k-2}{4-k}}\partial_\theta$ and $e_i, i=1,\cdots,m-1$ are unit tangent vectors of $\Sigma.$

Moreover, $\forall\ \alpha\in \mathbb{R}$, $\Sigma$ is $\mathcal{A}(\Sigma)=\int_{\Sigma}u^{\alpha}$-stable in $M$, i.e. $\forall\ \psi \in C_{0}^{\infty}(\Sigma)$, we have
\begin{equation*}
        \begin{split}
            \frac{d^2}{d t^2}\bigg|_{t=0}\mA(\Sigma_{t})&=\int_{\Sigma}|\nabla^{\Sigma}\psi|^2 u^{\alpha}-\left(|A_{\Sigma}|^2+\Ric_M(\nu,\nu)\right)\psi^2 u^{\alpha}\\
            &+\int_\Sigma \alpha(\alpha-1)u^{\alpha-2}\left(\nabla^M_{\nu}u\right)^2\psi^2 +\int_{\Sigma}\alpha u^{\alpha-1}\left(\Delta_M u-\Delta_{\Sigma} u\right)\psi^2\\
            &\geq 0.
        \end{split}
    \end{equation*}
\end{proposition}
\begin{proof}
    For convenience, we denote $\gamma=\frac{2}{4-k}$ and $\delta=\frac{4-2k}{4-k}$.
    It is not hard to check $\Sigma$ is totally geodesic in $M$ and $\Ric^{M}(e_{\theta},e_{\theta} )=-\frac{\Delta_{\Sigma}(u^{\delta})}{u^{\delta}}$. Thus, we have
    \[
    \begin{aligned}
    &\quad -k\Delta_{M}(u^{\gamma})+C_{m}^{M}(e_{1},...,e_{m-1},e_{\theta})u^{\gamma}\\
    &=-k\big(\Delta_{\Sigma}(u^{\gamma})+\frac{1}{u^{\delta}}\langle \nabla_{\Sigma}(u^{\delta}),\nabla_{\Sigma}(u^{\gamma})\rangle\big) +\big(C_{m-1}^{\Sigma}(e_{1},...,e_{m-1})-\frac{\Delta_{\Sigma}(u^{\delta})}{u^{\delta}}\big)u^{\gamma}   \\
    &\geq -k\big(\Delta_{\Sigma}(u^{\gamma})+\frac{1}{u^{\delta}}\langle \nabla_{\Sigma}(u^{\delta}),\nabla_{\Sigma}(u^{\gamma})\rangle\big) +\big(\frac{4}{4-k}\frac{\Delta_{\Sigma}u}{u}+\lambda-\frac{\Delta_{\Sigma}(u^{\delta})}{u^{\delta}}\big)u^{\gamma}\\
    &=\lambda u^\gamma
    \end{aligned}
    \]
    where in the inequality we have used $-\frac{4}{4-k}\Delta_{\Sigma}u+C_{m-1}^{\Sigma}u\geq \lambda u$.

Now we prove $\Sigma$ is weighted stable. Actually, Since $A_{\Sigma}=0$ and $u_{\nu}=0$, it suffices to show
\begin{equation*}
        \begin{split}
            \int_{\Sigma}|\nabla^{\Sigma}\psi|^2 u^{\alpha}-\Ric_M(e_{\theta},e_{\theta})\psi^2 u^{\alpha} +\int_{\Sigma}\alpha u^{\alpha-1}\left(\Delta_M u-\Delta_{\Sigma} u\right)\psi^2\geq 0.
        \end{split}
    \end{equation*}
Since $\Ric^{M}(e_{\theta},e_{\theta} )=-\frac{\Delta_{\Sigma}(u^{\delta})}{u^{\delta}}$ and $\Delta_{M}u=\Delta_{\Sigma}u+\frac{1}{u^{\delta}}\langle \nabla_{\Sigma}u^{\delta},\nabla_{\Sigma}u\rangle$, we only need to show
\begin{equation*}
        \begin{split}
            \int_{\Sigma}|\nabla^{\Sigma}\psi|^2 u^{\alpha}+\frac{\Delta_{\Sigma}(u^{\delta})}{u^{\delta}}\psi^2 u^{\alpha} +\int_{\Sigma}\alpha u^{\alpha-1}\left(\frac{1}{u^{\delta}}\langle \nabla_{\Sigma}u^{\delta},\nabla_{\Sigma}u\rangle\right)\psi^2\geq 0.
        \end{split}
    \end{equation*}
However, we obtain
\begin{equation*}
        \begin{split}
            &\quad\int_{\Sigma}|\nabla^{\Sigma}\psi|^2 u^{\alpha}+\frac{\Delta_{\Sigma}(u^{\delta})}{u^{\delta}}\psi^2 u^{\alpha} +\int_{\Sigma}\alpha u^{\alpha-1}\left(\frac{1}{u^{\delta}}\langle \nabla_{\Sigma}u^{\delta},\nabla_{\Sigma}u\rangle\right)\psi^2\\
            &=\int_{\Sigma}|\nabla^{\Sigma}\psi|^2 u^{\alpha}+\left(\Delta_{\Sigma}\log u^{\delta}+|\nabla_{\Sigma}\log u^{\delta}|^2\right)\psi^2 u^{\alpha} +\int_{\Sigma}\alpha\delta u^{\alpha-2}|\nabla_{\Sigma}u|^2\psi^2\\
            &=\int_{\Sigma}|\nabla^{\Sigma}\psi|^2 u^{\alpha}+|\nabla_{\Sigma}\log u^{\delta}|^2\psi^2 u^{\alpha} -\langle \nabla_{\Sigma}\log u^{\delta},2\psi\nabla_{\Sigma}\psi\rangle u^{\alpha}\\
            &\geq 0,
        \end{split}
    \end{equation*}
    where the second equation is due to integration by part.
\end{proof}

\subsection{Examples of non-splitting manifold}
Using Proposition \ref{lift C_m-1 to C_m}, we can show the dimension restriction of Theorem \ref{splitting of nonnegative C_m} is sharp by constructing examples with positive $m$-intermediate curvature (Theorem \ref{counterexample of splitting: main in the introduction} in the introduction).

\begin{theorem}\label{counterexample of splitting}
    When $6\leq n\leq 7$ and $2\leq m\leq n-3$, there exists a complete metric on $N^{n-m}\times T^{m-1}\times \mathbb{R}$ with uniformly positive $m$-intermediate curvature.
\end{theorem}
We first construct a manifold $N^{n-m}\times \mathbb{R}$ with a complete metric such that it has uniformly positive spectral $(\frac{2m-2}{m},1)$-intermediate curvature. For this, we consider the following example from \cite{XuTrans}, which shows that the range of $\gamma$ is sharp in Theorem \ref{splitting of nonnegative spectral ric}. Consider the manifold $\Sigma=\mathbb{S}^{n-m}\times \mathbb{R}$ with warped product metric $g=dr^2+\epsilon^2f(r)^2\bar{g}$, where $r\in(-\infty,+\infty)$ and $\bar{g}$ is the standard metric on $\mathbb{S}^{n-m}$. For $6\leq n\leq 7$ and $2\leq m\leq n-3$, we have $\frac{4}{n-m}\leq \frac{2m-2}{m}$. Thus, set $\lambda>0$, with the metric $g$ we would like construct the first eigenfunction $u=u(r)$ to satisfy 
\[
\frac{2m-2}{m}\Delta_{\Sigma}u= C_1^\Sigma(\partial_r)u-\lambda u
. 
\]
Here we can choose $\epsilon$ sufficiently small to make sure $C_1^\Sigma=C_1^\Sigma(\partial_r)$. We write down the explicit expressions of $u(r)$ and $f(r)$. Actually, $u(r)$ and $f(r)$ solve the following ODE
\begin{equation}\label{bottom spectral ode}
    \frac{2m-2}{m}\Big(\frac{u''(r)}{u(r)}+(n-m)\frac{f'(r)}{f(r)}\frac{u'(r)}{u(r)}\Big)=-(n-m)\frac{f''(r)}{f(r)}-\lambda.
\end{equation}
Solutions to the following system
\begin{equation}
    \begin{cases}
     (n-m)(\log f)'=h-\frac{2m-2}{m}(\log u)'\\
    \frac{2m-2}{m}(\log u)'=\frac{2}{2-n+m}h\\
    h'=-\lambda+\frac{n-m-\frac{4m}{2m-2}}{(n-m-2)^2}h^2
\end{cases}
\end{equation}
will solve equation \eqref{bottom spectral ode}.
If $\frac{4}{n-m}=\frac{2m-2}{m}$, then
\[
    u(r)=\exp \big(\frac{m\lambda r^2}{(2m-2)(n-m-2)}\big),\quad f(r)=\exp \big(-\frac{\lambda r^2}{2(n-m-2)}\big).
\]
If $\frac{4}{n-m}<\frac{2m-2}{m}$, then
\[
u(r) = \Big(\cosh\Big(\sqrt{C_{4}\lambda}r\Big)\Big)^{-\frac{mC_{3}}{(2m-2)C_{4}}},
\]
\[
f(r)=\Big(\cosh\Big(\sqrt{C_{4}\lambda}r\Big)\Big)^{\frac{C_{3}-1}{(n-m)C_{4}}},
\]
where $C_{3}=-\frac{2}{n-m-2}<0$ and $C_{4}=
\frac{(n-m)-\frac{2m}{m-1}}{(n-m-2)^2}>0$.
\vskip.2cm
We now prove Theorem \ref{counterexample of splitting}.
\begin{proof}[Proof of Theorem \ref{counterexample of splitting}]     
 We consider the manifold $M=\Sigma\times \mathbb{T}^{m-1}$ with metric $$g_{M}=dr^2+\epsilon^2f^{2}(r)\bar{g}+u^{\frac{4}{m}}(r)dx_{1}^{2}+\cdots +u^{\frac{4}{m}}(r)dx_{m-1}^2.$$
Denote $(\Sigma_{j},g_{j})=(\Sigma\times \mathbb{T}^{j},dr^2+\epsilon^2f^{2}(r)\bar{g}+u^{\frac{4}{m}}(r)dx_{1}^{2}+\cdots +u^{\frac{4}{m}}(r)dx_{j}^2)$. 

We have known from above that there exists a positive function $u$ on $\Sigma$ such that
\[-\frac{2m-2}{m}\Delta_\Sigma u+C_1^\Sigma(\partial_r)u=\lambda u.\]
First, using Proposition \ref{lift C_m-1 to C_m} for $\Sigma$, we get
\[
-\frac{2m-4}{m-1}\Delta_{\Sigma_{1}}u^{\frac{m-1}{m}}+C_{2}^{\Sigma_{1}}(\partial_{r},e_{x_{1}})u^{\frac{m-1}{m}} \geq \lambda u^{\frac{m-1}{m}}.
\]
Next, using Proposition \ref{lift C_m-1 to C_m} for $\Sigma_{1}$, we get
\[
-\frac{2m-6}{m-2}\Delta_{\Sigma_{2}}(u^{\frac{m-1}{m}})^{\frac{m-2}{m-1}}+C_{3}^{\Sigma_{2}}(\partial_{r},e_{x_1},e_{x_2})(u^{\frac{m-1}{m}})^{\frac{m-2}{m-1}} \geq \lambda (u^{\frac{m-1}{m}})^{\frac{m-2}{m-1}}.
\]
Continue this progress for $\Sigma_{j}$, we get
\[
-\frac{2m-2-2j}{m-j}\Delta_{\Sigma_{j}}u^{\frac{m-j}{m}}+C_{j+1}^{\Sigma_{j}}(\partial_{r},e_{x_1},...,e_{x_{j}})u^{\frac{m-j}{m}} \geq \lambda u^{\frac{m-j}{m}}.
\]
In particular, when $j=m-1$, we obtain
\[
C_{m}^{M}(\partial_{r},e_{x_{1}},...,e_{x_{m-1}})\geq \lambda. 
\]

Next we claim that $C_{m}^{M}=C_{m}^{M}(\partial_{r},e_{x_{1}},...,e_{x_{m-1}})$ when $\epsilon$ is sufficiently small. Once this is done, we construct a complete metric $g_M$ on $M=\Sigma\times \mathbb{T}^{m-1}=\mathbb{S}^{n-m}\times \mathbb{T}^{m-1}\times \mathbb{R}$ which is not splitting and satisfies $C_{m}^{M}\geq \lambda >0$.
Now we prove the claim whose proof is inspired by \cite[Theorem 1.4]{XuTrans}. We shall try to show detailed calculations. We take the local coordinate $(y_{\alpha},r,x_{i})$, then $g=\sum_{\alpha,\beta=1}^{n-m}\epsilon^2 f^2(r)h_{\alpha\beta}dy_{\alpha}dy_{\beta}+dr^2+\sum_{i=1}^{m-1}u^{\frac{4}{m}}(r)dx_{i}^{2}$, where $h_{\alpha\beta}$ is the standard metric on $\mathbb{S}^{n-m}$ which can be chosen to be diagonal. Let $e_\alpha=\epsilon^{-1}f^{-1}\partial_{y_\alpha},e_r=\partial_r,e_i=u^{-\frac{2}{m}}\partial_{x_i}$. Then $\{\{e_{\alpha}\}_{\alpha=1}^{n-m},e_{r},\{e_{i}\}_{i=1}^{m-1}\}$ is the local orthonormal frame of the metric $g$. Let indices $\{\alpha,\beta,\gamma\}$ be in the $\mathbb{S}^{n-m}$ direction. Let $i,j$ be in the $T^{m-1}$ direction. By direct computations, we have
\[
\Gamma^{r}_{\alpha\beta} = -\epsilon^2 f(r)f'(r)h_{\alpha\beta}, \quad \Gamma^{i}_{\alpha\beta}=0, \quad \Gamma^{\gamma}_{\alpha\beta}=\Gamma^{\gamma}_{\alpha\beta}(h), \quad \Gamma^{\beta}_{\alpha r}=\frac{f'(r)}{f(r)}\delta^{\beta}_{\alpha},
\]
\[
\Gamma^{i}_{\alpha r}=\Gamma^{r}_{\alpha r}=\Gamma^{\beta}_{\alpha i}=\Gamma^{j}_{\alpha i}=\Gamma^{r}_{\alpha i}=\Gamma^{\alpha}_{rr}=\Gamma^{r}_{rr}=\Gamma^{i}_{rr}=\Gamma^{\alpha}_{r i}=\Gamma^{r}_{r i}=\Gamma^{\alpha}_{ij}=\Gamma^{k}_{ij}=0,
\]
\[
\Gamma^{j}_{ir}=\frac{2u'(r)}{mu(r)}\delta^{i}_{j}, \quad \Gamma^{r}_{ij}=-\frac{2u'(r)}{mu(r)}u^{\frac{4}{m}}(r)\delta^{i}_{j}.
\]
Then we obtain
\[
\operatorname{Rm}(e_{\alpha},e_{\beta},e_{\alpha},e_{\beta})=\frac{1}{\epsilon^2f^2(r)}-\frac{(f'(r))^2}{f^2(r)}, \quad \operatorname{Rm}(e_{\alpha},e_{i},e_{\alpha},e_{i})=-\frac{2f'(r)u'(r)}{mf(r)u(r)}, 
\]
\[
\operatorname{Rm}(e_{\alpha},e_{r},e_{\alpha},e_{r})=-\frac{f''(r)}{f(r)},\] \[\operatorname{Rm}(e_{i},e_{r},e_{i},e_{r})=-\frac{2}{m}\frac{u''(r)}{u(r)}-\frac{2}{m}\left(\frac{2}{m}-1\right)\left(\frac{u'(r)}{u(r)}\right)^2,
\]
\[
\operatorname{Rm}(e_{\alpha},e_{l},e_{p},e_{q})=\operatorname{Rm}(e_{\alpha},e_{\beta},e_{\gamma},e_{l})=\operatorname{Rm}(e_{\alpha},e_{\beta},e_{p},e_{q})=0,
\]
\[ 
\operatorname{Rm}(e_{i},e_{j},e_{i},e_{j})=-\frac{4}{m^2}\Big(\frac{u'(r)}{u(r)}\Big)^2.
\]
Thus, for any vectors $ u,v$ that are tangent to $\mathbb{S}^{n-m}$, we have the following estimate
\begin{equation}\label{es1_S}
\begin{aligned}
\frac{1}{|u|^2}\Ric(u,u)
&=(n-m-1)\Big(\frac{1}{\epsilon^2f^{2}(r)}-\frac{(f'(r))^2}{f^2(r)}\Big)-\frac{f''(r)}{f(r)}-\frac{2(m-1)f'(r)u'(r)}{mf(r)u(r)}\\
&\geq \frac{n-m-1}{\epsilon^2f^2(r)}-C(n,m)\Big(\big|(\log f(r))''\big|+\big|(\log f(r))'\big|^2+\big|(\log u(r))'\big|^2\Big)
\end{aligned}
\end{equation}
and 
\begin{equation}\label{es2_S}
\operatorname{Rm}(u,v,u,v)\leq \frac{1}{\epsilon^2f^2(r)}\Big(|u|^2|v|^2-\langle u,v\rangle^2\Big).
\end{equation}

Similarly, for any vectors $ u,v$ that are tangent to $\mathbb{T}^{m-1}\times \mathbb{R}$, we have the following estimate
\begin{equation}\label{es1_T}
\begin{aligned}
\frac{1}{|u|^2}\Ric(u,u)
&\leq C(n,m)\Big(\big|(\log f(r))''\big|+\big|(\log f(r))'\big|^2+\big|(\log u(r))'\big|^2+\big|(\log u(r))''\big|\Big)
\end{aligned}
\end{equation}
and 
\begin{equation}\label{es2_T}
\operatorname{Rm}(u,v,u,v)\leq C(n,m)\Big(\big|(\log u(r))'\big|^2+\big|(\log u(r))''\big|\Big)\Big(|u|^2|v|^2-\langle u,v\rangle^2\Big).
\end{equation}
Now we suppose $\{X_{l}\}_{l=1}^{m}$ is a local orthonormal frame of the metric $g$. We decompose $X_l = X_l^T + X_l^S$, where $X_l^T$ are tangent to $\mathbb{R}\times \mathbb{T}^{m-1}$ and $X_l^S$ are tangent to $\mathbb{S}^{n-m}$.

Then 
\[
\begin{aligned}
&\quad C^M_m(X_{1},...,X_{m})\\
&=\sum_{l=1}^{m}\Ric(X_{l},X_{l})-\sum_{l<k}\operatorname{Rm}(X_{l},X_{k},X_{l},X_{k})\\
&=\sum_{l=1}^{m}\Ric(X_{l}^{T},X_{l}^T)+\sum_{l=1}^{m}\Ric(X_{l}^{S},X_{l}^S)\\
&-\sum_{l<k}\operatorname{Rm}(X_{l}^T,X_{k}^T,X_{l}^T,X_{k}^T)-\sum_{l<k}\operatorname{Rm}(X_{l}^S,X_{k}^S,X_{l}^S,X_{k}^S)-J\\
&=C_{m}^{M}(\frac{X_{1}^{T}}{|X_{1}^{T}|},\frac{X_{2}^{T}}{|X_{2}^{T}|},...,\frac{X_{m}^{T}}{|X_{m}^{T}|})-\sum_{l=1}^{m}\Ric(\frac{X_{l}^{T}}{|X_{l}^{T}|},\frac{X_{l}^{T}}{|X_{l}^{T}|})|X_{l}^{S}|^2\\
&+\sum_{l<k}\operatorname{Rm}(\frac{X_{l}^{T}}{|X_{l}^{T}|},\frac{X_{k}^{T}}{|X_{k}^{T}|},\frac{X_{l}^{T}}{|X_{l}^{T}|},\frac{X_{k}^{T}}{|X_{k}^{T}|})(1-|X_{l}^{T}|^2|X_{k}^{T}|^2)\\
&+\sum_{l=1}^{m}\Ric(X_{l}^{S},X_{l}^S)
-\sum_{l<k}\operatorname{Rm}(X_{l}^S,X_{k}^S,X_{l}^S,X_{k}^S)-J,
\end{aligned}
\]
where 
\[
\begin{aligned}
J
&=\sum_{l<k}\Big(\operatorname{Rm}(X_{l}^T,X_{k}^S,X_{l}^S,X_{k}^T)+\operatorname{Rm}(X_{l}^T,X_{k}^S,X_{l}^T,X_{k}^S)\Big.\\
&+\operatorname{Rm}\Big.(X_{l}^S,X_{k}^T,X_{l}^S,X_{k}^T)+\operatorname{Rm}(X_{l}^S,X_{k}^T,X_{l}^T,X_{k}^S)\Big).
\end{aligned}
\]
Then 
\begin{equation}\label{es_J}
|J|\leq C(n,m)\Big(\big|(\log f(r))''\big|+\big|(\log f(r))'\big|^2+\big|(\log u(r))'\big|^2\Big)\sum_{l}|X_{l}^S|^2,
\end{equation}
where we use $|X_{l}^T|\leq 1$, for $l=1,...,m$.

This is followed by \eqref{es1_S} \eqref{es2_S} \eqref{es1_T} \eqref{es2_T} and \eqref{es_J},
\[
\begin{aligned}
&\quad C^M_m(X_{1},...,X_{m})\\
&\geq C_{m}^{M}(\frac{X_{1}^{T}}{|X_{1}^{T}|},\frac{X_{2}^{T}}{|X_{2}^{T}|},...,\frac{X_{m}^{T}}{|X_{m}^{T}|})\\
&+ \frac{1}{\epsilon^2f^2(r)}\Big((n-m-1)\sum_{l=1}^{m}|X_{l}^{S}|^2-\sum_{l<k}\big(|X_{l}^{S}|^2|X_{k}^{S}|^2-\langle X_{l}^{S},X_{r}^{S}\rangle^2\big)\Big)\\
&-\tilde{C}(n,m)\Big(\big|(\log f(r))''\big|+\big|(\log f(r))'\big|^2+\big|(\log u(r))'\big|^2+\big|(\log u(r))''\big|\Big)\sum_{l=1}^{m}|X_{l}^{S}|^{2}.\\
\end{aligned}
\]
Note that
\[
\sum_{l<k}\big(|X_{l}^{S}|^2|X_{k}^{S}|^2-\langle X_{l}^{S},X_{r}^{S}\rangle^2\big)\leq \frac{1}{2}\big(\sum_{l=1}^{m}|X_{l}^{S}|^2\big)^2\leq \frac{n-m}{2}\sum_{l=1}^{m}|X_{l}^{S}|^2,
\]
where the last inquality is due to $\sum_{l=1}^{m}|X_{l}^{S}|^2\leq n-m$.
Hence, we finally obtain
\[
\begin{aligned}
&C_{m}(X_{1},...,X_{m})\geq C_{m}\Big(\frac{X_{1}^T}{|X_{1}^T|},...,\frac{X_{m}^T}{|X_{m}^T|}\Big)+\Big[(\frac{n-m}{2}-1)\epsilon^{-2}f^{-2}(r)\\
&-\tilde{C}(n,m)\Big(\big|(\log f(r))''\big|+\big|(\log u(r))''\big|+\big|(\log f(r))'\big|^2+\big|(\log u(r))'\big|^2\Big)\Big]\sum_{l=1}^{m}|X_{l}^{S}|^2.
\end{aligned}
\]
Since $f$ decays exponentially as $r\rightarrow \pm\infty$ and $\log f$, $\log u$ are of polynomial growth. Then the second term will be positive by taking $\epsilon$ small enough. Note that $n-m>2.$

\end{proof}

\subsection{Examples of asymptotically sharp diameter estimate}
In this section, we give another application of Theorem \ref{spectral inheritance} and Proposition \ref{lift C_m-1 to C_m} in diameter estimate of surfaces that is stable in certain sense.
\begin{theorem}\cite{xuguoyi-sharp-diameter-esitmate}\label{xuguoyi}
    Suppose $(M^3,g)$ has scalar curvature $R_g\geq 1$. Then any complete stable minimal surface $\Sigma$ in $M$ satisfies
    \[\operatorname{diam}(\Sigma)<\frac{2\sqrt{6}\pi}{3}.\]
    
    The upper bound is sharp in the following sense: there is a sequence of complete $3$-dimensional manifolds $M_k=(\mathbb{S}^2\times \mathbb{S}^1,g_k)$ with $R_{g_k}\geq 1$ and compact stable minimal surfaces $\Sigma_k\subset M_k$ such that $$\lim_{k\rightarrow \infty}\operatorname{diam}(\Sigma_k)=\frac{2\sqrt{6}\pi}{3}.$$
\end{theorem}
The inequality ($\leq\frac{2\sqrt{6}\pi}{3} $) was proved by Schoen-Yau \cite{Schoen-Yau-black-hole} while the asymptotical sharpness was proved by Hu-Xu-Zhang \cite{xuguoyi-sharp-diameter-esitmate}. In fact, they showed that there is a sequence of complete Riemannian manifolds $\Sigma_k=(\mathbb{S}^2,\bar{g}_k)$ with $\lambda_1(-\Delta_{\Sigma_k}+\beta K_{\Sigma_k})\geq \lambda$ for $\beta>\frac{1}{4}$ and $\lambda>0.$ Note that any stable minimal surface $\Sigma$ in $(M^3,g)$ with $R_g\geq 1$ satisfies $\lambda_1(-\Delta_{\Sigma}+ K_{\Sigma})\geq 1$.

\vskip.2cm
Our goal is to obtain a higher dimensional result by using proposition \ref{lift C_m-1 to C_m}.

\begin{definition}
Let $(M,g)$ be an $n$-dimensional Riemannian manifold with nontrivial homology group $H_k(M)$. The $k$-diameter of $(M,g)$ is defined as
\begin{equation*}
    \operatorname{diam}_k(M,g) := \inf \left\{ \operatorname{diam}_g(\Sigma) \;\middle|\;
\begin{aligned}
&\Sigma \subset M \text{ is a smooth submanifold}, \\
&[\Sigma] \neq 0 \text{ in } H_k(M)
\end{aligned} \right\},
\end{equation*}
where $\operatorname{diam}_g(\Sigma)$ denotes the diameter of $\Sigma$ with respect to the induced metric.
\end{definition}
\begin{theorem}\label{diameter inequality in Cm case}
    Assume either \(3 \leq  n \leq  5\), \(1 \leq  m \leq  n-1\) or \(6 \leq  n \leq  7\), \(m\in \{1,n-2,n-1\}\). Suppose $(N^{n},g)$ is a closed Riemannian n-manifold with positive $m$-intermediate curvature $C_{m}\geq \lambda>0$. Suppose there exists a proper smooth map $f:N\rightarrow M^{n-m+1}\times \mathbb{T}^{m-1}$ with nonzero degree. Then
    $$\operatorname{diam}_{n-m+1}(N,g)\leq  \frac{\pi}{\sqrt{\lambda C_{0}}}$$ 
    where
\[
C_{0}=\frac{m^2-mn+m+n}{2(m^2-nm+2n-2)}>0.
\] 
In particular,  when $m\geq 2$, the inequality is strict unless $n=m+2$.
\end{theorem}

This recovers the case $m=2$ and $3\leq n\leq 5$, proved by Shen–Ye \cite{shenyingyerugang} (see Theorem \ref{Diameter estimate of Shen-ye}). The same upper bound was also obtained by Wu \cite[Theorem 1.5]{wuyujie-diameter-estimate} using weighted minimal slicings from Brendle–Hirsch–Johne \cite{brendlegeroch'sconjecture}. We present a direct proof here using Theorem \ref{spectral inheritance}.
\begin{proof}
    The dimension $n\leq 7$ ensures the regularity of minimizers. The case $m=1$ follows from Bonnet-Meyer's theorem. By assumption, we know $N$ has positive $(0,m)$-intermediate curvature, i.e., $C_m^N\geq\lambda.$ First we find an area minimizer $\Sigma^{n-1}$ (closed manifold) in the homology class $f^{*}[M^{n-m+1}\times \mathbb{T}^{m-2}\times \{p\}]$. By Theorem \ref{spectral inheritance}, $\Sigma^{n-1}$ admits positive spectral $(1,m-1)$-intermediate curvature. We then continue to find minimizers of $\mathcal{A}_k$ and apply Theorem \ref{spectral inheritance} by $m-2$ more times. In the end, we can construct a $(n-m+1)$-dimensional closed manifold $\Sigma^{n-m+1}$ which admits positive spectral $\frac{2m-2}{m}$-Ricci curvature.  It is not hard to check that 
    \begin{equation}\label{condition 1}
        \frac{2m-2}{m}< \frac{4}{n-m}.
    \end{equation}
    Then using Theorem \ref{Diameter estimate of Shen-ye}, we obtain 
    \[
    \operatorname{diam}(\Sigma^{n-m+1})\leq \frac{\pi}{\sqrt{\lambda C_{0}}}.
    \]
    Note that $[\Sigma^{n-m+1}]=[N]\frown (f^*(d\theta_{1})\smile\cdots \smile f^{*}(d\theta_{m-1}))\neq 0$. Then
    \[
    \operatorname{diam}_{n-m+1}(N,g)\leq \frac{\pi}{\sqrt{\lambda C_{0}}}.
    \]
    The proof is completed.
\end{proof}

In the following, we study the sharpness of the inequality in Theorem \ref{diameter inequality in Cm case}.

\begin{theorem}\label{example sharp diameter of C_m}
Let $3\leq n\leq 7$ and $m\in\{n-1,n-2\}$. There exists a sequence of smooth complete $n$-dimensional closed Riemannian manifold $N^{k}$ with positive $m$-intermediate curvature $C_{m}\geq \lambda>0$ and a sequence of smooth $(n-m+1)$-dimensional submanifold $\Sigma^{k}\subset H_{n-m+1}(N^{k};\mathbb{Z})$ satisfying $$\operatorname{diam}(\Sigma^{k})\rightarrow \frac{\pi}{\sqrt{\lambda C_{0}}}$$ increasingly as $k\rightarrow \infty$, where
\[
C_{0}=\frac{m^2-mn+m+n}{2(m^2-nm+2n-2)}>0.
\]
\end{theorem}

\begin{remark}
     For $m=n-2$, the theorem is trivial since  $\mathbb{S}^{3}(\sqrt{\frac{2}{\lambda}})\times \mathbb{T}^{n-3}$ with canonical metric achieves the upper bound of diameter. We only need to prove the case $m=n-1.$
\end{remark}

\begin{proof}
    Using Theorem \ref{xuguoyi}, we can construct a sequence of spheres $(\Sigma_{k},g_{k})=(\mathbb{S}^{2},dr^2+ f_{k}^{2}(r)\bar{g}_{\mathbb{S}^{1}})$ with uniformly positive spectral $(\frac{2n-4}{n-1},1)$-intermediate curvature (i.e. $\lambda_{1}(-\frac{2n-4}{n-1}\Delta_{\mathbb{S}^2}+K)\geq \lambda$) and $\operatorname{diam}(\Sigma_{k},g_{k})\rightarrow \frac{\pi}{\sqrt{\lambda C_{0}}}$ as $k\rightarrow \infty$. This sequence of spheres geometrically converges to a segment. Then we consider the manifold $N=\Sigma_{k}\times \mathbb{T}^{n-2}$ with metric $$g_{k}=dr^2+f^{2}_{k}(r)\bar{g}_{\mathbb{S}^1}+u_{k}^{\frac{4}{n-1}}(r)dx_{1}^{2}+\cdots +u_{k}^{\frac{4}{n-1}}(r)dx_{n-2}^2.$$
Then using Proposition \ref{lift C_m-1 to C_m}, we get
\[
R^{M_{k}}=2C_{n-1}^{M_{k}}(\partial_{r},e_{x_{1}},...,e_{x_{n-2}})\geq 2\lambda. 
\]     
Apparently, the slice $\{x_{1}=p_{1},\cdots,x_{n-2}=p_{n-2}\}\subset N^k$ is nontrivial in $H_{2}(N^k;\mathbb{Z})$.
\end{proof}

It remains an interesting question whether Theorem \ref{example sharp diameter of C_m} holds for other choices of $m\geq 2.$

\bibliographystyle{alpha}

\bibliography{reference}

\end{document}